\documentclass[10pt,a4paper]{article}
\usepackage{mathrsfs, geometry}
\usepackage[english]{babel}
\usepackage{enumitem}
\usepackage{graphicx}
\usepackage{epstopdf}
\usepackage{epsfig,psfrag}
\usepackage{amsmath, amsfonts, amsthm,amssymb}
\usepackage{accents}
\usepackage{hyperref}
\usepackage{aurical}
\usepackage[T1]{fontenc}
\usepackage{xcolor}
\usepackage{yfonts}

\usepackage{pifont}
\DeclareFontFamily{U}{astrosym}{}
\DeclareFontShape{U}{astrosym}{m}{n}{<-> astrosym}{}
\DeclareFontFamily{U}{dancers}{}
\DeclareFontShape{U}{dancers}{m}{n}{<-> dancers}{}
\newtheorem{theorem}{Theorem}[section]
\newtheorem{definition}[theorem]{Definition}
\newtheorem{lemma}[theorem]{Lemma}

\newtheorem{corollary}[theorem]{Corollary}
\newtheorem{remark}[theorem]{Remark}

\newtheorem{proposition}[theorem]{Proposition}

\numberwithin{equation}{section}

\newcommand{\yy}{{\mathcal Y}}
\newcommand{\s}{{\sigma}}
\newcommand{\id}{{\rm Id}}
\newcommand{\jj}{{j}}

\renewcommand{\Im}{{\mbox{Im}}}
\renewcommand{\Re}{{\mbox{Re}}}

\newcommand{\wtS}{\widetilde{S}}

\newcommand{\ad}{{\rm ad}}
\newcommand{\tT}{{\mathtt{T}}}

\DeclareMathAccent{\wtilde}{\mathord}{largesymbols}{"65}

\newcommand{\C}{{\mathbb C}}

\newcommand{\N}{{\mathbb N}}

\newcommand{\R}{{\mathbb R}}

\newcommand{\T}{{\mathbb T}}
\newcommand{\Z}{{\mathbb Z}}
\newcommand{\VV}{{{\mathscr V}}}
\newcommand{\PP}{{{\mathscr P}}}
\newcommand{\cA}{{\mathcal A}}

\newcommand{\cC}{{\mathcal C}}

\newcommand{\cG}{{\mathcal G}}

\newcommand{\cK}{{\mathcal K}}
\newcommand{\cL}{{\mathcal L}}

\newcommand{\cN}{{\mathcal N}}
\newcommand{\cO}{{\mathcal O}}

\newcommand{\cR}{{\mathcal R}}
\newcommand{\cS}{{\mathcal S}}
\newcommand{\cT}{{\mathcal T}}

\newcommand{\cV}{{\mathcal V}}

\newcommand{\cY}{{\mathcal Y}}

\newcommand{\fA}{{\mathfrak{A}}}
\newcommand{\fa}{{\mathfrak{a}}}

\newcommand{\fH}{\mathfrak{H}}

\newcommand{\fL}{{\mathfrak{L}}}
\newcommand{\fM}{{\mathfrak{M}}}

\newcommand{\fS}{{\mathfrak{S}}}
\newcommand{\fT}{{\mathfrak{T}}}

\newcommand{\ta}{{\mathtt{a}}}

\newcommand{\tc}{{\mathtt{c}}}
\newcommand{\td}{{\mathtt{d}}}

\newcommand{\tp}{{\mathtt{p}}}

\newcommand{\tj}{{\mathtt{j}}}
\newcommand{\tk}{{\mathtt{d}}}

\newcommand{\tR}{{\mathtt{R}}}

\newcommand{\tJ}{{\mathtt{J}}}

\newcommand{\tL}{{\mathtt{L}}}

\newcommand{\bb}{{\bf b}}

\newcommand{\bk}{{\bf k}}

\newcommand{\bv}{{\bf v}}

\newcommand{\bN}{{\bf N}}
\newcommand{\bK}{{\bf K}}

\newcommand{\al}{{\alpha}}
\newcommand{\bt}{{\beta}}

\newcommand{\g}{\gamma}
\newcommand{\f}{\varphi}

\newcommand{\norm}[1]{\| #1 \|}

\newcommand{\abs}[1]{| #1 |}
\newcommand{\0}{{(0)}}

\newcommand{\la}{\left\langle}
\newcommand{\ra}{\right\rangle}

\newcommand{\im}{{\rm i}}
\newcommand{\jap}[1]{\langle #1 \rangle}
\newcommand{\und}[1]{\underline{#1}}

\newcommand{\e}{{\varepsilon}}

\newcommand{\meas}{{\rm meas}}

\newcommand{\diag}{{\rm diag}}

\definecolor{darkgr}{rgb}{0.0, 0.62, 0.42}

\newcommand{\bnorm}[1]{{|\mkern-6mu |\mkern-6mu | \,  #1 \,  |\mkern-6mu 
		|\mkern-6mu |}  }

\newcommand{\sega}[3]{\abs{#1}^{\mathtt {qT}}_{#2,#3}}
\newcommand{\nor}[3]{\abs{#1}_{#2;#3}}
\newcommand{\segm}[3]{\abs{#1}^{\mathtt {qT},#2}_{#3,-m}}
\newcommand{\seuno}[3]{\abs{#1}^{\mathtt{qT},#2}_{#3,-1}}
\newcommand{\sedue}[3]{\abs{#1}^{\mathtt {qT},#2}_{#3,-2}}
\newcommand{\op}[2]{\cL^{\mathtt {qT},#1}_{#2,-m}}
\newcommand{\opuno}[2]{\cL^{\mathtt {qT},#1}_{#2,-1}}
\newcommand{\opdue}[2]{\cL^{\mathtt {qT},#1}_{#2,-2}}
\definecolor{cobalt}{rgb}{0.0, 0.28, 0.67}

\title{Reducibility and nonlinear stability for a quasi-periodically forced NLS}

\author{E. Haus\footnote{Dipartimento di
	Matematica e Fisica,
	Universit\`a di Roma Tre, Largo S.L. Murialdo, 00100 Roma, Italy,
	{\tt ehaus@mat.uniroma3.it}}  \, 
B. Langella\footnote{International School for Advanced Studies (SISSA), Via Bonomea 265, 34136, Trieste, Italy
	{\tt beatrice.langella@sissa.it}}	, \,
A. Maspero\footnote{International School for Advanced Studies (SISSA), Via Bonomea 265, 34136, Trieste, Italy
	{\tt alberto.maspero@sissa.it}}, \,
 M. Procesi\footnote{Dipartimento di Matematica e Fisica,
	Universit\`a di Roma Tre, Largo S.L. Murialdo, 00100 Roma, Italy,
	{\tt procesi@mat.uniroma3.it}}
}

\begin{document}
	\maketitle
	
	\begin{abstract}
	Motivated by the problem of long time stability vs. instability of KAM tori of the Nonlinear cubic Schr\"odinger equation (NLS) on the two dimensional torus $\T^2:= (\R/2\pi \Z)^2$, we consider a quasi-periodically forced NLS equation  on $\T^2$ arising from the linearization of the NLS at a KAM torus. 
	We prove a reducibility result as well as  long time stability of the origin.
	The main novelty is to obtain the  precise  asymptotic expansion of the frequencies which allows us to impose Melnikov conditions at arbitrary order.
	\end{abstract}
	
	\section{Introduction and main results}
	The nonlinear cubic Schr\"odinger equation on the two dimensional torus $\T^2:= (\R/2\pi \Z)^2$
	\begin{equation}\label{NLS0}
	\im \partial_t v = - \Delta v + |v|^2 v , \quad x \in \T^2 \ 
	\end{equation}
	is one of the fundamental equation in mathematical physics.
	The pioneering work by Bourgain
	\cite{Bourgain93} ensures that the equation is  locally well-posed for any data in $H^s(\T^2)$ for all $s >0$, where 
	$$
\|u\|_{H^s} :=\left(\sum_{n \in \Z^2} (1+|n|)^{2s}|\widehat u_n|^2\right)^{\frac12}.
$$
	Much less is known on the long time behaviour of the solutions, and in the last ten years  a rich and diverse dynamics  has been discovered for \eqref{NLS0}.
Such dynamics comprise both ``regular'' dynamics, such as periodic \cite{GP3}, quasiperiodic \cite{EK10,GYX,PP13}, almost-periodic \cite{Bou,BMP:almost} trajectories contained in invariant tori, and more ``irregular'' one, such as 
weakly turbulent trajectories undergoing energy cascade phenomenons \cite{CKSTT,GuardiaK12,GuardiaHP16,Hani12,GHHMP22}.\\
A particularly interesting long term goal is to understand the interplay of these two behaviours, in particular whether there are weakly turbulent trajectories connecting distinct invariant tori. \\
An important intermediate step is to understand the local dynamics close to these invariant sets, namely if one takes an initial datum close to an invariant torus, for how long does its trajectory stay close to the torus? The next step is to show that, beyond the stability times, unstable phenomena occur, and one may explicitly construct orbits which start $\delta$-close to an invariant torus and then slowly drift away by a finite but arbitrarily large factor.  A convenient way of ensuring that trajectories do not stay close to a given torus is to show that they exhibit a Sobolev norm explosion, namely that  after some (very long) time their  Sobolev norm, in some fixed $H^s(\T^2)$ with $s\ne 0,1$,  becomes larger than an arbitrary number  $K$ (recall that on a given invariant torus the Sobolev norms are essentially constant).

The simplest invariant object is clearly the fixed point $v = 0$. 
In this case stability results were proved for example in 
\cite{BG,Faou-Grebert:2013,BMP:2019, BerGre,BerGre2}, while instability results in  \cite{Bourgain96,Kuksin96, Kuksin97, Kuksin97b}, the landmark result in \cite{CKSTT} and subsequent generalizations \cite{GuardiaK12,HP,GuardiaHP16}.
For these problems  a fundamental role is played by  the linearized frequencies at $ v = 0$ of \eqref{NLS0}, namely
\begin{equation}\label{freq:0}
\Omega_j = |j|^2 , \quad j \in \Z^2 , 
\end{equation}
and the interest is either to avoid resonances between them  to prove stability results or to  exploit them to create instability phenomenon.

A much more challenging problem is to study stability/instability of more complicated invariant objects.
%
The simplest nontrivial invariant objects are the invariant tori of dimension 1 filled with  plane wave solutions of the form  $u(t,x) = \rho e^{\im j x - \omega t}$, $\omega = |j|^2 + \rho^2$. Their long time stability was proved in
\cite{Faou-etc}, whereas their instability in $H^s$, $s \in (0,1)$, in \cite{Hani12}.
Again a fundamental role is played by the resonance properties of the linearized frequencies at plane waves solutions, which are easily computed to be 
\begin{equation}
\label{freq:PW}
\Omega_j = \sqrt{|j|^4 + 2 |j|^2 \rho^2} = |j|^2 + \frac{c_j(\rho)}{|j|^2} , \quad j \in \Z\setminus \{0\} \ . 
\end{equation}
 Here it is important to notice how the correction to the unperturbed frequencies \eqref{freq:0} decays in $|j|$, namely whatever is the direction of the vector $j \in \Z^2$.
 
The next studied invariant objects have been  the so-called finite gap solutions, which are solutions for the integrable 1d cubic NLS. 
In particular, they are solutions of  \eqref{NLS0} depending only on 1 variable, say $x_1$, but they are quasi-periodic in time and fill invariant tori of finite dimension $\td \in \N$.
Their long time stability has been studied in \cite{MP} whereas their instability, again in $H^s$ with $s \in (0,1)$ in \cite{GHHMP22}.
As the reader might guess, again a fundamental role is played by the linearized frequencies at finite gap solutions. 
In this case it is quite involved to  compute their expression, and it  requires partial Birkhoff normal form, a sort of ``1d pseudodifferential normal form'' and KAM reducibility techniques. 
The final result is that they expand as
\begin{equation}
	\label{asinto}
\Omega_j = |j|^2 + \mu(j) + \frac{\fa(m)}{\jap{m}^2} + \frac{R_j}{\jap{j}^2} , \qquad j = (m,n) \in \Z^2 \setminus \{0\} \ , 
\end{equation}
where $\mu(j)$ has a finite number of possible values while $\fa,R$ are of size $\e$.
Note the first difference with \eqref{freq:0} and \eqref{freq:PW}: 
the correction to the unperturbed frequencies depends on the direction of the vector $j = (m,n) \in \Z^2$.

\vspace{.5em}
\noindent{\bf Towards stability/instability of KAM tori.} 
The next natural step is to  study the long time stability vs instability of more general (truly bidimensional) quasiperiodic tori of 
Procesi-Procesi \cite{PP13}. 
There are a number of very serious difficulties to overcome in order to deal with this more general case.
Again the main challenge is to obtain the correct asymptotic expansion of the  linearized frequencies at KAM tori, namely the analogous of formulas \eqref{freq:PW},  \eqref{asinto}.
This is a very hard problem which requires new ideas with respect to the finite gap case.

The goal of this paper is to settle the machinery to obtain such asymptotics and to apply them to  study a simplified NLS-like equation  which,
despite not having a particular physical interest, contains the main difficulties to be overcome in order to tackle the full problem of stability/instability of KAM tori.

In order to justify the model and the simplifications that we will introduce, let us first give a brief description of how the KAM-solutions appear.
The first step is to choose an appropriate approximate solution,  typically a solution of the linearized equation at $v=0$, which is the starting point of a quadratic algorithm which converges to a true solution. The idea in \cite{GYX,PP,W11} is to fix a finite set $\cS=\{\bk^{(1)},\dots,\bk^{(d)}\}\subset \Z^2$ (of maximal rank),  look for solutions of  \eqref{NLS0} close to 
\[
v_{\rm lin}(t,x)= \sum_{j\in \cS} \sqrt{\xi_j} \, e^{\im j\cdot x + \im |j|^2 t}\,,
\]
and show that for {\em most} choices of $\cS$ and {\em many} choices of $\xi\in \R_+^d$ (in some small ball of radius $\e$) there exists a true solution  $q^{\rm qp}$ of \eqref{NLS0} essentially supported on the Fourier modes in $\cS$ which is quasi-periodic in time with frequency 
$$
\omega = (\omega_i)_{i=1}^d , \qquad \omega_i  = |\bk^{(i)}|^2 + \cO(\e)\,.
$$
Since the NLS equation is Gauge and translation invariant it turns out that the solutions in \cite{PP13} have the form of {\em covariant quasi-periodic traveling waves}, according to the following definition:

\begin{definition}[Quasi-periodic traveling waves] 
\label{def:tw}
Let 
 $\bK\in$ Mat$_{d\times 2}(\Z)$ (of maximal rank) with 
\begin{equation}
	\label{Kappa}
\bK^T= \begin{pmatrix}
	\bk^{(1)}_1 & \dots & \bk^{(d)}_1 \\\bk^{(1)}_2 & \dots & \bk^{(d)}_2 
\end{pmatrix} \ . 
\end{equation}
 A function $q(\f, x)$ is called a {\em quasi-periodic traveling wave} if it has the form 
$$
q(\f, x):= Q(\f + \bK x) 
$$
where $Q\colon \T^d \to \C$ is a real analytic function.\\
We say that a traveling wave is Gauge covariant  if $Q$ is a function with Fourier coefficients supported in the set $ \{ \ell \colon \ell \cdot \vec{1} = 1\}$, i.e.
\begin{equation}
Q(\f + t \vec{1})  = e^{\im t } Q(\f) , \qquad \vec 1=(1,\dots,1)\in \Z^d \ . 
\end{equation}
 We say it is Gauge invariant if   supported in the set $ \{ \ell \colon \ell \cdot \vec{1} = 0\}$.
Finally we set
	\begin{equation}
	\label{normap}
	\|Q\|_{\ta, \tp}^2:= {\sum_{ \ell } e^{2\ta|\ell|}|Q(\ell)|^2 \jap{\ell}^{2\tp}}  < \infty\,, \quad\mbox{for}\quad \ta>0,\;\tp>\frac d2 \ .
\end{equation}
\end{definition}
Note that  Gauge covariant, quasi-periodic traveling waves satisfy 
\begin{align}
\label{cov00}
q(\f + \bK \zeta, \cdot) = \tau_\zeta q(\f, \cdot) \ , \quad \forall \zeta \in \R^2 \\
\label{cov01}
q(\f + t \vec{1}, \cdot)  = e^{\im t } q(\f, \cdot)  \ , 
\quad \forall t \in \R \,,
\end{align}
In this sense, the solutions in \cite{PP13} are covariant quasiperioric traveling waves of the form
\begin{equation}
	\label{QP}
	q^{\rm qp}(\omega t,x) = q^{\rm qp}(\f,x)\vert_{\f = \omega t} = Q(\omega t + \bK x)
\end{equation}
where  $\omega\in \R^d$ belongs to a Cantor-like set of positive
measure and $\bK$ as in \eqref{Kappa}.

Linearizing \eqref{NLS0} at  $ q^{\rm qp}$, one gets the equation
\begin{equation}
	\label{NLS1}
	\im \partial_t u = - \Delta u + 
	A(\omega t) u + {\cal N}(\omega t, u)
\end{equation}
with 
\begin{equation}\label{AN}
\begin{aligned}
	A(\f) u:=    2|q^{\rm qp}(\f,\cdot)|^2 u + q^{\rm qp}(\f, \cdot )^2 \bar u \ , \\
	{\cal N}(\f, u):=  2 q^{\rm qp}(\f,\cdot) |u|^2 + \bar q^{\rm qp} (\f,\cdot) u^2 + |u|^2 u \ .
\end{aligned}
\end{equation}
As the original equation \eqref{NLS0} is gauge and translation {co}variant and $ q^{\rm qp}(\f, x)$ is a gauge covariant  quasi-periodic traveling wave, the linear operator $A(\f)$ and the nonlinearity $\cN(\f, u)$ fulfill the following covariant properties, as we show in Appendix \ref{section:cov}:  for any
$(\f, \zeta, t)   \in \T^d \times \R^2 \times \R $ 
\begin{align}
&	A(\f + t \vec{1}) \circ e^{\im t } = e^{\im t } \circ A(\f) , \qquad \qquad 
	A(\f +  \bK \zeta) \circ \tau_\zeta = \tau_\zeta \circ A(\f) , \label{covA}
\\
&	\cN(\f + t \vec{1},  e^{\im t} u)  = e^{\im t } \circ \cN(\f,u) , \qquad
	\cN(\f +  \bK \zeta, \tau_\zeta u ) = \tau_\zeta \circ \cN(\f,u)  \, ,  \label{covN}
\end{align}
where $\tau_\zeta$ is the translation operator $(\tau_\zeta u)(x):= u(x + \zeta)$.

 The fact that one may reduce to constant coefficients the operator $A(\f)$ was already discussed in \cite{PP,PP13} (and is actually simpler to prove than in the case of finite gap solutions \cite{MP}).
  On the other hand the methods of \cite{PP,PP13} do not give any information on the asymptotics of the frequencies, and actually one can prove that an expansion as the one in \eqref{asinto} cannot hold. 
Then one is not able to prove that the non-resonance conditions hold and one may not prove neither the stability result nor the successive instability one.

\smallskip

\paragraph{The results.}
As we already mentioned, the goal is this paper is to initiate the analysis of stability/instability of the general case by studying the nonlinear stability problem for a simplified {\em non-resonant} model, which contains only the difficulties related to the Melnikov conditions and avoids the further algebraic complications (which  one expects to deal with exactly as in \cite{PP} and \cite{MP}).
To this purpose we study the following simplified model:
\begin{equation}\label{NLS3} 
	\im u_t	- \Delta u + \VV(\omega t,x) u +\partial_{\bar u} \PP(t,x, u,\bar u)=0\,,\qquad x\in \T^2 \ . 
\end{equation}
We make the following assumptions on $\VV$ and $\PP$:
\begin{itemize}
\item[(H$\VV$)]   $   \VV(\omega t, x)$ is the multiplication operator by an 
 analytic  gauge invariant traveling wave (see  Definition \ref{def:tw}), namely
 \begin{equation}
	\label{travel}
{	\mathscr V(\omega t,x)}= V(\omega t + \bK x)  \,,\quad V(\f)= \sum_{ \ell \in \Z^d} V(\ell) e^{\im\ell\cdot \f}\,,
\end{equation}
with the constraint that
\begin{equation}
\label{condV}
V(\ell)=0\quad  \forall \ell: \; \pi(\ell)=0\,,\quad \mbox{or}\quad \sum_{i=1}^d \ell_i \ne 0\,. 
\end{equation}
\item[(H$\PP$)] The non linear term 
 $\PP(t,x,y_1,y_2) = P(\omega t +\bK x, y_1,y_2)$ with $P$  an  analytic function on $\T^d\times B_{\mathtt r_0}(0)$,
\[
P(\f,y_1,y_2)= \sum_{\ell, d_1,d_2} P_{\ell,d_1,d_2} e^{\im \ell\cdot \f} y_1^{d_1}y_2^{d_2}\,,\quad  \| P\|_{\ta,\mathtt r_0}=  \sum_{\ell, d_1,d_2} |P_{\ell,d_1,d_2}| e^{\ta|\ell|} \mathtt r_0^{d_1+d_2}<\infty \ , 
\]
with a zero of order at least two 
 and satisfying the invariance property
\[
P(\f+t \vec 1, e^{\im t} y_1, e^{-\im t} y_2 ) =P(\f,y_1,y_2)\,.
\]
\end{itemize} 
\begin{remark}
We model $\VV$ on the first term of the operator $A(\f)$ in \eqref{AN}, namely  the multiplication operator by $|q^{\rm qp}(\omega t, x)|^2$, which is a gauge invariant traveling wave with 
$ \int_\T |q^{\rm qp}(\omega t, x)|^2 dx $ constant in time. So we  assume that the same applies to $\VV$.
\end{remark}

\begin{remark}
The operator  $\VV(\f,\cdot)$ and the nonlinearity $\PP(\f,\cdot)$ fulfill \eqref{covA}, \eqref{covN}.
\end{remark}

For any $p \geq 0$ we introduce the space $\mathfrak H^p = \fH^p(\Z^2):= \ell^1(\Z^2)\cap h^p(\Z^2)$ endowed with the norm
$$
\| u \|_p := \| u \|_{\ell^1(\Z^2)} + \| u \|_{h^p(\Z^2)}.
$$
Note that, by Young inequality, the space $\mathfrak H^p$ is an  algebra with respect to the convolution product, i.e.
\begin{equation}\label{algebra}
\| u * v \|_p \leq 2^{2p+1} \|u \|_p \, \|v \|_p
\end{equation}
We are now ready to state our stability result.
	\begin{theorem}\label{cubic}
		Fix $p>0$ and consider the equation \eqref{NLS3}. Assume (H$\VV$) and  (H$\PP$). Let $\tR := [-1,1]^d$.
		There exist $\gamma_*, \epsilon_*>0$ and for any $\gamma \in (0, \gamma_*)$, and any potential $V$ satisfying
		\begin{equation}\label{pic colo0}
		\epsilon:=	\gamma^{-1}\,  \norm{V}_{\ta, \tp} < \epsilon_*  \ , 
		\end{equation}
		there exist a Cantor-like set
		$\cC$ with measure 
		$\meas (\tR \setminus \cC) \leq C \gamma\,,$ 
		and 
		$r_*,  T_* >0 $ such that { for any $\omega \in \cC$,}  $r \in (0, r_*)$ the following holds true. 
		For any $u \in \fH^p$, with $\norm{u(0)}_{p} \leq r$, the solution $u(t) \in \fH^p$ of \eqref{NLS3} fulfills
		\begin{equation}
			\label{}
			\norm{u(t)}_p \leq 2 r , \qquad \forall |t| \leq \frac{T_*}{r^2}
			\ . 
		\end{equation}
	\end{theorem}
	{
		As already mentioned, the proof of Theorem \ref{cubic} follows the same general scheme of \cite{MP}  based on the three steps: 1.  Reducibility, 2. Non-resonance conditions, 3. Order three  Birkhoff Normal Form.
	Since the last step is standard (it is briefly discussed in Section \ref{3birk}) we concentrate on the first two.
\\
	We start by studying the  reducibility for the Schr\"odinger operator 
	\begin{equation}\label{Lo}
			A= \im \fL\,,\quad 	\fL = -\Delta + \VV\,,
			\end{equation}
	 in the space $\mathfrak H^p := \ell^1(\Z^2)\cap h^p(\Z^2)$, where $p>0$ is given, which amounts to prove that the following property holds:
	\begin{definition}[Reducibility]\label{def.red}
		We say that a time dependent Schr\"odinger operator $\fL$  is reducible on  $\fH^p$ if  there exists a  quasi-periodic in time bounded and invertible linear operator
		$G$  such that
		\begin{equation}\label{G.star.L}
		G_* \fL:= -\im G^{-1} \dot G + G^{-1}\fL G  = \diag(\Omega_j)\,,\quad \Omega_j\in \R\,.
		\end{equation}
		\end{definition}
	In \cite{EK1}, Eliasson and Kuksin have proven in the analytic setting the existence of an invertible and time quasi-periodic linear operator $G$ such that $G_* \fL$ as in \eqref{G.star.L} is reduced to a time independent, block diagonal operator $Q$, under smallness assumptions on $\epsilon$ as in \eqref{pic colo0} and for a Borel set of frequencies $\omega$ of asymptotically full measure in $\epsilon$, without the constraint that $\VV$ is a quasi-periodic traveling wave. Furthermore, it is well known (see for instance \cite{PX}) that if $\fL$ is covariant in the sense of \eqref{covA}, then $Q$ is diagonal. However, the reducibility result of \cite{EK1} does not provide an asymptotic expansion of the frequencies $\{\Omega_j\}_{j \in \Z^2}$. Thus here
	we proceed to prove a more refined result, from which we obtain an asymptotic expansion for $\{\Omega_j\}_{j \in \Z^2}$ that generalizes the one found in \cite{GHHMP22} in the finite gap case (see \eqref{asinto}).
	 To state our result we first need two definitions.
	\begin{definition}[Generators]\label{def.generators}
		A vector $v=(v_1,v_2) \in \Z^2\setminus\{0\}$ shall be called a \emph{generator} if
		\begin{equation}\label{def:generators}
		{\rm gcd}(v_1,v_2)=1 \mbox{ and }  v_1>0  \quad \mbox{ or } \quad v = (0,1) \ . 
		\end{equation}
		We shall denote by $\cV\subset\Z^2\setminus\{0\}$ the set of generators. 
	\end{definition}
	From now on, we fix 
	\begin{equation}
	\label{mu}
	0<\delta\ll 1 , \quad \mbox{ and } \quad  \mu := 1-2\delta   \ .
	\end{equation}
	\begin{definition}\label{def:VjBj}
		Given $j\in \Z^2\setminus\{0\}$, we define $v(j)\in \cV$ as the first vector w.r.t.\ the lexicographic ordering that attains the minimum in
		\begin{equation}
		\label{vmin}
		\min_{v\in \cV: |v|\le  |j|^\delta}|v\cdot j|
		\end{equation}
		and set
		$b(j):= j\cdot v(j)$. 	This means that $\,|b(j)|= \min \,\{|v\cdot j|\ |\ v\in \cV: |v|\le  |j|^\delta\}$.
	\end{definition}
	With these notations, we prove the following:
	\begin{theorem}\label{riduco.circa}
	There exist $\gamma_*, \epsilon_*>0$ and for any $\gamma \in (0, \gamma_*)$,  any potential $V$ satisfying \eqref{pic colo0},
	there exist a Cantor-like set
	$\cC$ with measure 
	$\meas (\tR \setminus \cC) \leq C \gamma \,, $
	a sequence of Lipschitz functions $\mathtt R\ni \omega \mapsto \Omega_j(\omega)$ and a bounded linear operator $G$, close to the identity and with quasi-periodic dependence on time, such that {for any $\omega \in \cC$}
	\begin{equation}
	\label{cerca.circa}
	G_* \fL= \diag(\Omega_j)\,,
	\end{equation}
	with
{\begin{equation}\label{nice.expansion}
	\begin{gathered}
\Omega_j(\omega) :=	\Omega_j = |j|^2 + \frac{\varpi(v(j),b(j) , \omega)}{\jap{b(j)}^2}{e^{- \kappa|v(j)|}} + \frac{\Theta^{(1)}(j, \omega)}{\jap{b(j)}\jap{j}^{\mu}} + \frac{\Theta^{(2)}(j, \omega)}{\jap{j}^{2\mu}} \quad \forall j\,,\\
\sup_{\omega \in \tR} \, 	\sup_{j \in \Z^2} \left(|\varpi(b(j), v(j), \omega)| + |\Theta^{(1)}(j,\omega )| + |\Theta^{(2)}(j,\omega )|\right) < 2 \|V\|_{\ta,\tp}\,,
	\end{gathered}
	\end{equation}}
	with $\kappa>0$ depending on $\ta$ and on the matrix $\bK$ only. Moreover, for any integer $N \geq 3$, there exist $\tau_N>0, c^*_{N}>0$ and a Cantor set $\cC_N \subseteq \cC$, such that
	$\meas(\tR \backslash \cC_N) \leq c^*_N \gamma\,,$
	and $\forall \omega \in \cC_N$, $\forall (\ell, L) \in \cG_N$, 
	where
	\begin{multline}\label{G.N}
	\cG_N := \Big\{(\ell,L)\in \Z^d\times \Z^{\Z^2} :\; \mbox{if} \;  \ell=0 \,,\quad \mbox{then}\quad \sum_{j\in \Z^2}|j|^2 L_j\ne 0\\
	|L|:= \sum_{j\in \Z^2}|L_j|\le N\,, \quad \; \sum_{i=1}^d \ell_i+  \sum_{j\in \Z^2}L_j=0\,,\quad
	\sum_{i=1}^d\ell_i\bk^{(i)} +  \sum_{j\in \Z^2}j L_j= 0 \Big\}\,,
	\end{multline}
	one has
	\begin{equation}\label{N.melnikov}
	|\omega\cdot\ell + \Omega\cdot L |\ge \g \jap{\ell}^{-\tau_N}\,.
	\end{equation}
	\end{theorem}
	The major step in proving non-resonance conditions \eqref{N.melnikov} is to show that the final frequencies $\{\Omega_j\}_{j \in \Z^2}$ satisfy the asymptotic expansions \eqref{nice.expansion}. Actually the problem  of exhibiting asymptotic expansions for the eigenvalues of a linear Schr\"odinger operator $$
	\widetilde{\fL} = -\Delta + W
	$$ on $\T^n$, $n \geq 1,$ has been widely investigated already in the case of a time independent $W$ (see for instance \cite{KFTst, KFTunst, Veliev, Kar96, Kar97, PS10, PS12, BLMspec20, BLMspec22}). Roughly speaking, in such works one partitions the spectrum of $\fL$ into two subsets: the stable spectrum, which is composed by the eigenvalues such that
	\begin{equation}\label{stable}
	\Omega_j = |j|^2 + O\left(\frac{1}{\jap{j}^{2\mu}}\right)\,,
	\end{equation}
	and the unstable one, namely the set of eigenvalues for which \eqref{stable} does not hold. We point out that with our notations, stable eigenvalues correspond to those such that 
	\begin{equation}\label{big.b}
	\jap{b(j)} \geq \jap{j}^{\mu}\,.
		\end{equation}
	Refined asymptotic expansions were available in the time independent case also for eigenvalues in the unstable spectrum, and in \cite{BLMspec22} it was proven that all the eigenvalues of $\widetilde{\fL}$ on $\T^2$ are of the form
	\begin{equation}\label{csi.M}
	\begin{gathered}
	\Omega_j = |j|^2 + \sum_{n =1}^N \frac{{\varpi}^{\ast}_n(j)}{\jap{b(j)}^{2n}} + \sum_{n =1}^N \frac{{\Theta}^{\ast}_n(j)}{\jap{j}^{2n\mu}} + O\left(\frac{1}{\jap{b(j)}^{2N}}\right)\,,
	\end{gathered}
	\end{equation}
	using techniques based on pseudo-differential calculus and a geometric decomposition of the space of indexes $j$ \`{a} la Nekhoroshev (see \cite{Nekhoroshev77, Giorgilli:2003}). However, an expansion of the form \eqref{csi.M} would not suffice in order to show non-resonance conditions \eqref{N.melnikov} at any order $N$: this is the reason why we use a slightly different approach, based on a structure which we refer to as quasi-T\"oplitz (see Definition \ref{def.top} below), and prove expansions \eqref{nice.expansion}. Remark that the major difference with respect to the expansions  \eqref{csi.M} is represented by the functions $\varpi(b(j), v(j))$, namely by the fact that the terms which decay with $\jap{b(j)}$ depend on the vector $j$ through $v(j)$ and $b(j)$ only and decay exponentially in $v(j)$. This is the key to proving the non-resonance conditions \eqref{N.melnikov}.

Indeed, rewriting  
	\[
	\omega \cdot \ell + \Omega\cdot L = \omega \cdot \ell  \pm  \Omega_{j_1} \pm \Omega_{j_2}\pm \dots \pm  \Omega_{j_k} 
	\]
	the asymptotic expansion guarantees that if all the $|j_i| \gg |\ell|^{a_1}$,  and either all the $|b(j_i)| \gg |\ell|^{a_2}$ or $|v(j_i)| \gg |\ell|,$  then \eqref{N.melnikov} follows by imposing  Diophantine conditions on $\omega$ of the form
	$$
	|\omega \cdot \ell + K| \gtrsim |\ell|^{-\tau} \quad \forall \ell \neq 0\,, \quad \forall K \in \Z\,.
	$$
	(and it is well known that such conditions hold for a positive measure set of $\omega$ provided that $\tau$ is small enough).
	\\
	Following this line of reasoning in \eqref{N.melnikov} we can ignore all  the terms $\frac{\Theta^{(1)}(j_i)}{\jap{b(j_i)}\jap{j_i}^\mu}$, $\frac{\Theta^{(2)}(j_i)}{\jap{j_i}^{2\mu}}$ where the $|j_i|$ are sufficiently large. Similarly we can ignore all the terms $\frac{{\varpi}(v(j_i),b(j_i)}{\jap{b(j_i)}^2} e^{-\kappa |v(j_i)|}$ where either $v(j_i)$ or $b(j_i)$ is large.  This means that, for $|L|\le N$, the expression $\omega \cdot \ell + \Omega\cdot L$ assumes (up to a negligible error)  only a finite number of distinct values (of cardinality bounded only by $\ell$) which are small.
	\\
	In turn this means that one can surely impose conditions of the type \eqref{N.melnikov} provided that $\tau_N$ is sufficiently large.

	\section{Functional setting}

\normalsize

\subsection{Time-dependent momentum preserving operators.}
We shall consider quasi-periodic time dependent operators $t\mapsto M(\omega t)\in \cL(\fH^p)$ with frequency $\omega\in \R^d$, where 
\begin{equation}
\label{M(t)}
 M(\f)= \sum_{\ell\in \Z^d} M(\ell) e^{\im \f \cdot \ell   }\,,\quad M(\ell)\in \cL(\fH^p),
\end{equation}
that are analytic and momentum preserving  according to the following definition:
\begin{definition}[Momentum preserving operators]
	\label{def:momentum}
	Here and in the following we shall define the linear map $\pi:\Z^d\to \Z^2$
	\begin{equation}\label{ccc}
	\pi(\ell) :=  \sum_{i=1}^d\ell_i\bk^{(i)} \,,\quad 	\tc^{-1}:= \sup_{\ell\ne 0} \frac{|\pi(\ell)|}{|\ell|}\,.
	\end{equation}
	We say that a  time dependent linear operator $M(\omega t)$ is {\em momentum preserving} if
	\begin{equation}
	\label{memento}
	j - j'\neq \pi(\ell)   \ \ \Rightarrow \ \ 
	M_j^{j'}(\ell):= \langle M(\ell) e^{\im j'\cdot x},  e^{\im j\cdot x} \rangle =0 \ . 
	\end{equation}
	For any $a  > 0$ we define the norm
	\[
	|M|_{a}:= \sup_{\|u\|_p \leq 1 } \|\underline{M}_{a} u\|_p \,,\quad (\underline{M}_{a})_{j}^{j'}:= \sum_{ \ell: \pi(\ell)= j-j'} e^{a|\ell|} |M_{j}^{j'}(\ell)| . 
	\]
	We denote the space of time dependent momentum preserving operators with finite norm as $\cL_a(\mathfrak H^p)$. 
	\end{definition}
	\begin{definition}[Gauge covariant operators]
	We say that $M$ as above is {\em Gauge covariant} if $\sum_{i=1}^d \ell_i\ne 0$ implies $M_j^{j'}(\ell)=0\,.$
	\end{definition}
	
A convienent way of envisioning such operators is as {\em normally analytic} maps from a thickened torus
	\begin{equation}
			\label{toro}
		\T^d_a:=\{\f\in \C^d:\quad \Re(\f)\in \R^d /  (2\pi \Z)^d\,,\quad |\Im(\f)|\le a\}\quad \to \cL(\fH^p)\,.
	\end{equation}
\begin{remark}
	If $M \in \cL_a(\mathfrak H^p)$, then  $M(\f) \in \cL(\mathfrak H^p)$ for all $\f\in \T^d_a$ and the operator norm satisfies ${\| M(\f) \|_{\cL(\mathfrak H^p)} \leq |M|_a.}$
\end{remark}
{One may easily verify (see Lemma \ref{lemma.algebra}) that $\cL_a(\fH^p)$ is in fact an algebra with respect to composition. Then it is standard to define the commutator and the adjoint action as follows:
	\begin{definition}
		For $M,S\in \cL_a(\fH^p)$, we set 
		$
		\ad S[M](\f):= [M(\f), S(\f)]:= M(\f) S(\f)- S(\f) M(\f)$, which in matrix form reads
		\[
		([M, S])_j^{j-\pi(\ell)}(\ell)= \sum_{\ell_1+\ell_2=\ell} M_j^{j-\pi(\ell_1)}(\ell_1) S_{j-\pi(\ell_1)}^{j-\pi(\ell)}(\ell_2)- S_j^{j-\pi(\ell_1)}(\ell_1) M_{j-\pi(\ell_1)}^{j-\pi(\ell)}(\ell_2)\, . 
		\]
	\end{definition}
}
{
	\begin{definition}\label{def.sad-symp}
		We say that $M\in \cL_a(\fH^p)$ is self-adjoint if it satisfies 
		\begin{equation}
		\label{AA}
		M_{j}^{j-\pi(\ell)}(\ell) = \overline{ M_{j-\pi(\ell)}^{j}(-\ell)}
		\end{equation}
		so that   $M(\f)$ is self-adjoint   for all  $\f\in \T^d$.
		As is standard, we say that a bounded  operator $G\in \cL_a(\fH^p)$  is symplectic if 
	$G=e^{\im  A}$ with $A\in \cL_a(\fH^p)$ self-adjoint.
	\end{definition}
	\begin{remark}\label{gruppo}
		If $A,B$ are  self-adjoint, then so is $\im [A,B]$.  
		Consequently, anti self-adjoint operators  form a Lie algebra. Moreover, if $M$ is self-adjoint and $G= e^{S}$ is symplectic  (i.e., $\im S$ is self-adjoint)  then 
		$
		G^{-1} M G = \exp(\ad(S)) M ${is self-adjoint}.
		Similarly { Gauge  covariant operators form a Lie algebra, so if $S,M$ are  Gauge covariant, so is $G$ and $G^{-1} M G$. }
\end{remark}}
	\paragraph{Time-dependent changes of variables}
Given the linear PDE $u_t = \im \fL u$,
where $\fL$ is a time dependent Schr\"odinger operator as in \eqref{Lo} with $P\in \cL_a(\fH^p)$, let $G\in \cL_a(\fH^p)$ be a quasi-periodic in time  bounded invertible change of variables of the form $G= e^S$ with $S\in \cL_a(\fH^p)$. Then, defining $u =: G v$, one has
$v_t = \im (G_*\fL ) v\,,$
with $G_* \fL$ as in \eqref{G.star.L}.
The Lie exponentiation formula gives a nice representation of $G_* \fL$ in terms of the adjoint actions:
\begin{equation}
\label{coniugio}
\fL_1 := G_*\fL =  e^{\ad S } \fL - \im \sum_{k= 1}^\infty \frac{(\ad S)^{k-1}}{k!} \dot S\,,\quad e^{\ad S } \fL:=  \sum_{k= 0}^\infty \frac{(\ad S)^{k}}{k!} \fL\,.
\end{equation}
{\begin{remark}
		Note that, by Remark \ref{gruppo}, if $\fL$ is self-adjoint and $G$ is symplectic (or equivalently $\im S$ is self-adjoint) then $\fL_1$ is self-adjoint.
		{ The same holds for the Gauge covariance.}
		\end{remark}}
{\paragraph{Symplectic structure}
As we have already explained in the introduction, time dependent Schr\"odinger operators have a natural Hamiltonian structure on the phase space $\T^d \times\R^d\times  \fH^p$, equipped with the symplectic form
$d\cY\wedge d\f + \im \sum_j d u_j\wedge d \bar u_j\,.$
To this purpose 
we associate  to $\fL$ as in \eqref{Lo}, with $P$ self-adjoint,  the Hamiltonian
		\begin{equation}
			\label{NLSHam}
			H_{\fL}:=  \omega\cdot \cY + \sum_j |j|^2 |u_j|^2 +  \sum_{j,\ell} P_j^{j-\pi(\ell)}(\ell) e^{\im \ell\cdot \f} u_{j-\pi(\ell)}\bar u_j\,,
		\end{equation}
	with Hamilton equations
 $
 \dot\f=\omega, \ \dot u =\im \fL u\,, \dot \cY= - (u, P_\f u)_{\ell_2(\C)}
 $.
Accordingly, to a symplectic map  $G= e^{\im A}$ we associate the generating function
\[
\cA(\f,u):=   \sum_{j,\ell} A_j^{j-\pi(\ell)}(\ell) e^{\im \ell\cdot \f} u_{j-\pi(\ell)}\bar u_j\,,
\] 
whose time one flow $\Phi_{\cA}^1$ gives the symplectic change of variables $(u,\f,\cY)= \cG(v,\f',\cY')=\Phi_{\cA}^1(v,\f',\cY')$
\begin{equation}
	\label{simplettica}
\f= \f' \,,\quad \cY  
=  \cY'+ \im (G v, G_\f v)_{\ell_2(\C)} \,, \quad u =G(\f) v\,.
	\end{equation}
Note  that (by the Lie exponentiation formula)  $\Phi_{\cA}^1$ conjugates
$
(\Phi_{\cA}^1)_* H_\fL = e^{\{\cA,\cdot\}} H_\fL = H_{\fL_1},
$
with $\fL_1$ defined in \eqref{coniugio}. }
{Finally if $P, A$ are Gauge covariant  then so is $\fL_1$.}
\paragraph{Basic properties of $\cL_a(\fH^p)$.}
We now list some useful properties of the space $\cL_a(\fH^p)$.
\begin{remark}\label{rem:maj.norm}
	We shall systematically use the fact that the norm $| \cdot |_a$ of Definition \ref{def:momentum} is ordered. Indeed, if $ M(\f), N(\f)$ are such that $(\underline{M}_a)^{j'}_j \leq  (\underline{N}_{a'})^{j'}_j $ for any $j,j'$, then $|M|_a \leq |N|_{a'}$.
\end{remark}

\begin{lemma}\label{gigetto}
	Let $M$ be a momentum preserving operator according to Definition \ref{def:momentum}.  If there exist $\theta, A>0$ such that
	\[
	(\underline{M}_{a})_{j}^{j'}\leq A \,  e^{-\theta|j-j'|} \,,\quad \forall j,j'\in \Z^2,
	\]
	then there exists $C(p) >0$ such that 
	\[
	|M|_{a} \leq C(p) A \,   \theta^{-p-2}\,.
	\]
\end{lemma}
\begin{proof}
	We remark that $\| \underline{M}_a u\|_p \leq \| f * u\|_p$ where $f = (f_j)_{j \in \Z^2}$ has components $f_j := e^{-\theta |j|}$. Then one uses \eqref{algebra} and the bound follows. 
\end{proof}

\begin{lemma}
	Let $\VV$ be a traveling wave as in Definition \ref{def:tw} with $\|V\|_{\ta, \tp}<\infty$.
	Then the multiplication operator  $M_V: u \mapsto \VV u$ belongs to  $\cL_a(\mathfrak H^p)$ for any $ a \in [0,\ta)$ and 
	\[
	|M_V|_{a} \leq  C(p,\tp) \, (\ta - a)^{-(p-2)} \, \|V\|_{\ta, \tp}\,.
	\]
\end{lemma}
\begin{proof}
	The multiplication operator  is represented by the matrix 
\begin{equation}
		\label{def:Mv}
	(M_V)_j^{j'}(\ell) = V_{\ell, j-j'} = 
	\begin{cases} 
	V(\ell) &  \mbox{ if }  j-j' = \pi(\ell) \\
	0  & \mbox{otherwise}
	\end{cases}  \ , 
\end{equation}
	where we used the fact that $V$ is a traveling wave, see \eqref{travel}.
	Hence $M_V$ is momentum preserving.
	To compute its norm, remark that $|j-j'|=|\pi(\ell)| \le \tc^{-1} |\ell|$ hence, setting $\theta =\tc(\ta-a)$,
	\begin{align*}
	\label{cestoaprova}
	( \und{M_V}_a)_j^{j'} &=\sum_{ \ell : j-j'= \pi(\ell)} e^{a|\ell|}|V(\ell)| \le e^{-\theta|j-j'|} \sum_{\ell} e^{\ta|\ell|}|V(\ell)|\\
	&\le e^{-\theta|j-j'|} \sqrt{ \sum_{ \ell } \frac{1}{\jap{\ell}^{2\tp}}}\sqrt{\sum_{ \ell } e^{2\ta|\ell|}|V(\ell)|^2 \jap{\ell}^{2\tp}} \le  A(\tp) e^{-\theta|j-j'|}\|V\|_{\ta,\tp}  \ .
	\end{align*}
	Then use Lemma \ref{gigetto}.
\end{proof}

{ We now want to define the "order" of a momentum preserving operator (see Definition \ref{def.due.pesi} below). On the one hand, an operator of order $-m$ is standardly defined as a linear map $\fH^p \rightarrow \fH^{p+m}$ $\forall p$; on the other hand, if $\fH^p$ represents a space of functions of two variables, $\fH^p(\Z^2) = \fH^p(\Z) \otimes \fH^p(\Z),$ one can define a "vectorial order" such that an operator of order $-(n_1, n_2)$ maps
$$
\fH^p(\Z) \otimes \fH^p(\Z) \rightarrow \fH^{p + n_1}(\Z) \otimes \fH^{p +n_2}(\Z)\,
$$
(see for instance \cite{MP}). Here actually the main novelty is that we use a "non standard" set of coordinates, which we define in the following.
} 	
\subsection{A non-linear coordinate set on $\Z^2$}
{In this section we prove some properties of the quantities $b(j)$, $v(j)$ defined for all $j \in \Z^2$ in Definitions \ref{def.generators}, \ref{def:VjBj}.}

\begin{lemma}\label{lemma:cramer}
	Let $v,w\in\cV$ be distinct generators with $\max(|v|,|w|) < R$. If $x\in\Z^2$ satisfies
	$\max(|x\cdot v|,|x\cdot w|)< A$, then $|x| < 2 A R$.
\end{lemma}
\begin{proof}
Denoting by $a:=(x \cdot v,  x \cdot w)^T$, $x$ is the solution of the linear system $M x = a$ where $M$ is the matrix with rows $v^T$ and $w^T$. Then one estimates $M^{-1}$ by Cramer's rule.
\end{proof}

We shall denote by  $ B_K(0)$ the ball of radius $K$ and center $0$ in $\Z^2$.
\begin{lemma}\label{traviata}
	There exists $\tJ_\delta>0$ with the following property. For any $j\notin  B_{\tJ_\delta}(0)$  such that
	$|b(j)|< 2 |j|^{\mu}$ (see \eqref{mu}), the following holds true:
	\begin{enumerate}
		\item[(i)] For all $w\in\cV$ with $|w|\le  |j|^\delta$ and $w\ne v(j)$
		\[
		|w\cdot j|\ge 2 \jap{j}^\mu > |v(j)\cdot j|\,,
		\]
		namely $v(j)$ is the unique vector attaining the minimum in \eqref{vmin}.
		\item[(ii)]  If there exists $h\in \Z^2$ with $|j-h|< 2 \max(|j|,|h|)^\delta$ and  $v(h)\ne v(j)$, then
		\[
		|v(j)|>\frac12 |j|^\delta.
		\]
	\end{enumerate}
\end{lemma}

\begin{proof}
$(i)$ Assume by contradiction that there exists another generator $w\in\mathcal V$ such that $|w\cdot j| = |v(j)\cdot j| = |b(j)| < 2|j|^{\mu}$.
	Then, we apply Lemma \ref{lemma:cramer} with $A=  2|j|^{\mu}, R= |j|^\delta$ and we deduce
	$$
	|j| < 4 |j|^{1-\delta} \quad \Rightarrow \quad |j| < 4^{\frac{1}{\delta}}
	$$
	and item $(i)$ follows provided $\tJ_\delta \geq 4^{\frac{1}{\delta}}$.
	
$(ii)$ We claim that, by taking  $\tJ_\delta$ large enough, one has
\begin{equation}
\label{h.est.1}
	\frac12 |j|<|h|< 2|j|.
\end{equation}
	Indeed if $|h| \geq 2|j|$, then by triangular inequality one deduces that $|h| \leq 4 |h|^\delta$, so in particular, being $\delta \ll 1$, one has $|h|\leq R_\delta$, same for $|j|$. Then just take $\tJ_\delta > R_\delta$ to get a contradiction. The other inequality is analogous.

	By the definition of $v(\cdot)$ we have
$	|v(j)|< |j|^\delta$, $|v(h)|< |h|^\delta$.
	We distinguish two cases.
	If $|v(j)|> |h|^\delta$ then the thesis follows by \eqref{h.est.1}.
	If $|v(j)|\le |h|^\delta$, then using the bounds on  $ b(j)$, $|j-h|$ and again \eqref{h.est.1} we get 
	\[
	|v(h)\cdot h|\le  |v(j)\cdot h | = |v(j)\cdot (j+ h-j)| \le 2 |j|^{\mu}+ 2 |j|^\delta \max(|j|,|h|)^\delta < 4 |h|^{\mu}
	\]
	provided $\tJ_\delta$ is sufficiently large and $\delta$ sufficiently small.

	Then we apply Lemma \ref{lemma:cramer} with $x = h$, $v= v(j)$, $w= v(h)$ $R= |h|^\delta$ and $A= 4|h|^{\mu}$. We deduce $|h|< 8 |h|^{1-\delta}$, which leads to a contradiction for $\tJ_\delta$ large.
\end{proof}

Next we define the order of a momentum preserving operator.
\begin{definition}[Order]\label{def.due.pesi}
	Given $\bN= (n, m )$ with $n,m \geq 0$ and an operator $M\in \cL_a(\fH^p)$, we define the following norm:
	\begin{equation}\label{due.pesi}
		|M|_{a; -\bN} := \sup_{\|u\|_p\leq 1} \|\underline{M}_{a; -\bN } u\|_p\,, \quad \left(\underline{M}_{a; -\bN}\right)_{j}^{j'} := \sum_{\ell \colon  j - j' = \pi(\ell)} e^{a|\ell|}\,  |M_{j}^{j'}(\ell)| \, \langle j \rangle^{\mu n} \, \langle b(j) \rangle^{m}\,,
	\end{equation}
	where $\forall j \in \Z^2$,  $b(j)$ is the quantity  in Definition \ref{def:VjBj} and $\mu$ in \eqref{mu}.
	We denote by $\cL_{a,-\bN}$ the subspace of $\cL_a$  with finite $|\cdot |_{a,-\bN}$ norm. 
\end{definition}

These operators form an algebra, as the next lemma shows. Let
\begin{equation}\label{def.C}
	C(p,\s):= c^{-p} \sup_{k\in \N} e^{-\s k} k^{p} \sim  \s^{-p} \, .
\end{equation}
\begin{lemma}[Algebra property]\label{lemma.algebra}
	Given two operators  $M, N$, with $|M|_{a; -\bN_1}, |N|_{a;-\bN_2}<\infty\,,$  
	we have
	\begin{equation}
		\label{stima1}
		|M N|_{a;-\bN_1} \le |M|_{a; -\bN_1} |N|_{a;\vec 0 }\,.
	\end{equation}
	Moreover, let $\bN_1 = (n_1, m_1)$ and $\bN_2 = (n_2, m_2)$; then we have for any $a'<a$
	\begin{equation}
		\label{stima2}
		| M N|_{a'; -(n_1 + n_2, 0)} \le  2^{n_2}|M|_{a';-\bN_1}|N|_{a'; -\bN_2}  +   C((n_1+n_2)/\delta, a-a') |M|_{a; -\bN_1}|N|_{a'; -\bN_2}\,.
	\end{equation}
\end{lemma}

The proof, being quite technical, is postponed to Appendix \ref{sec:technical}.
{\begin{remark}
	We note that if $M\in \cL_a(\fH^p)$ is time independent, then (by formula \eqref{memento}, and recalling that $\pi(0) = 0$) it is necessarily diagonal, namely there exists a sequence $\{M_j\}_{j\in\Z^2}$, $M_j\in \C$, with $M= \diag(M_j)_{j\in \Z^2}$. Moreover one has
	$\ \sup_{j\in \Z^2}|M_j| = |M|_a \,,\quad \forall a\,.$
\end{remark}}

\subsection{Quasi-T\"oplitz norm}
{The purpose of this section is to define, for a fixed $m>0$, a decomposition of operators in  $\cL_a(\fH^p)$ into a sum of terms of order $-\bN=-(m-k, k)$ for all $k = 0, \dots, m$. We shall further require an invariance property on the term of order $-(0, m)$, which we will refer to as line-T\"oplitz.}
\begin{definition}\label{def.top}
	An operator    $M\in\cL_a(\fH^p)$ is said to be: 
	\begin{itemize}
		\item {\em Line-T\"oplitz} if there exists a map
		$\Z^d\times\cV\times \Z \to \C\,,$ $(\ell,v,b) \mapsto \mathfrak M(\ell,v,b)\,,$
		such that 
		\[
		M_j^{j'}(\ell) = \fM(\ell,v(j),b(j))  , \quad \forall \ell \in \Z^d, \ j, j' \in \Z^2 \ . 
		\]
		We denote by  $\cT_a:=\cT_a(\fH^p)$ the set of  {line-T\"oplitz} operators in $\cL_a(\fH^p)$.
		\item {\em Line-T\"oplitz of order $-m$}, $m \geq 0$, if 
	{
		\begin{equation}\label{norm.toplitz}
			| M|^{\mathtt{ T}}_{a,-m}:=  \sup_{ v\in \cV\,,  \  b\in\Z}\sum_{ \ell\in\Z^d } e^{a \tc |v|+a |\ell|}|\mathfrak M(\ell,v,b) | \jap{b}^{m} < + \infty 
		\end{equation}
	where $\tc$ is defined in \eqref{ccc}.	}
		We denote by $\cT_{a,-m}$, 	 $m\ge 0$, the set of  
		line-T\"oplitz operators of order $m$.
		
		\item {\em Quasi-T\"oplitz  of order $-m$}, $m \in \N$,  if there exist a line-T\"oplitz operator $M^{\tT} \in \cT_{a,-m}$ 
		and $m$ operators $M^{(i)}\in \cL_{\frac{a}{2}}(\fH^p)$, $i = 1, \ldots, m$ such that
		\begin{equation}
		\label{arco}
		M = M^{\tT }+ \sum_{i=1}^{m}M^{(i)}\,,  \quad M^{(i)} \in \cL_{\frac{a}{2}; -(i, m-i)}\ . 
		\end{equation}
		We denote by  $\cL^{\mathtt{qT}}_{a,-m}$ the set of quasi-T\"oplitz  operators of order $m$ of $\cL_a(\fH^p)$ , which we endow  with the norm
		\[
		\sega M {a}{-m}:= \inf \left\{|M^{\tT}|^{\mathtt {T}}_{a,-m} +   \sum_{i = 1}^{m}|M^{(i)}|_{\frac{a}{2}; -(i, m-i)}\quad  \left|  \quad \, M = M^{\tT} + \sum_{i=1}^{m} M^{(i)}\right.\right\}\,.
		\]
	\end{itemize}
\end{definition}

{
\begin{lemma}\label{vaccab}
	The line-T\"oplitz operators are bounded, more precisely for all $a'\in(0,a)$ one has
	\[
	|M|_{a';-(0,m)}\le C(\tc,p) (a-a')^{-p-2} |M|^{\mathtt{ T}}_{a,-m}
	\]	
\end{lemma}
\begin{proof}
	We have\footnote{As is conventional we write $A\lesssim B$ if there exists a universal constant $C>0$ such that $A\le C B$. Similarly we write $A\lesssim_{\tc,p} B$ if there exists  $C>0$ depending only on $\tc,p$ such that $A\le C B$.}
	$$
	(\underline{M}_{a';-(0,m)})_j^{j'} = \sum_{\ell : j-j'=\pi(\ell)} e^{a|\ell|} \jap{b(j)}^m |\mathfrak M(\ell,v(j),b(j))| e^{-(a-a')|\ell|} \lesssim  |M|^{\mathtt{ T}}_{a,-m} e^{-\tc(a-a')|j-j'|}.
	$$
	Let $g_j:= e^{-\mathtt c(a-a')|j|}$ and $g:=\{g_j\}_{j\in\Z^2}\in\mathfrak H^p$. 
	We have
	\[
	|g|_p:= \sqrt{\sum_{j} e^{-2\tc(a-a')|j|} \jap{j}^{2p} } + \sum_{j} e^{-\tc(a-a')|j|}  \lesssim_{\tc,p} (a-a')^{-p-2}
	\]
	It follows from the algebra property of $\mathfrak H^p$ w.r.t.\ convolution that
	$$
	|M|_{a'; -(0, m)}\lesssim |g|_p |M|^{\mathtt{ T}}_{a,-m} \lesssim_{\tc,p}  (a-a')^{-p-2} |M|^{\mathtt{ T}}_{a,-m}\,.
	$$
\end{proof}
}

We shall need to keep track of Lipschitz dependence on the parameter $\omega \in \R^d$. To this purpose, given a compact set $\cO \subset \R^d$, we fix $\g>0$ and define the following norm on  the space of  Lipschitz maps  $f: \cO\to E$ ($E$ a Banach space) 
\[
| f|_E^\cO:= \sup_{\omega\in \cO}|f(\omega)|_E + \g  \sup_{\omega\ne \omega'\in \cO}|\Delta_{\omega,\omega'} f |_E\,,\qquad \Delta_{\omega,\omega'} f := \frac{f(\omega)-f(\omega')}{|\omega-\omega'|}\,.
\]
\begin{definition}
	Let $\cO\subset\R^d$ be compact. We denote by $\op{\cO}{a}$ the set of Lipschitz maps   $M:\cO\to \cL^{\mathtt {qT}}_{a; -m}$  with the  $\segm{\cdot}{\cO}{a}$ norm. 
	\[
	\segm{M}{\cO}{a}:= \sup_{\omega\in \cO}\sega{M(\omega)}{a}{-m} + \g  \sup_{\omega\ne \omega'\in \cO}\sega{\Delta_{\omega,\omega'} M}{a}{-m}
	\]
\end{definition}

Given a momentum preserving operator $M\in\cL_a(\mathfrak H^p)$ and $K>0$, we define the projections $\Pi_{|\ell|\le K} M$, $\Pi_{|\ell| > K} M$ by
\begin{equation}\label{proiezioni}
\Pi_{|\ell|\le K} M= \begin{cases}
	M_j^{j-\pi(\ell)}(\ell)
	\quad & \mbox{if}\; |\ell|\le K\\
	0 & \mbox{otherwise}
	\end{cases}, \qquad \Pi_{|\ell| > K} M := M - \Pi_{|\ell|\le K} M.
\end{equation}
These projections obviously map the space $\cL_a(\mathfrak H^p)$ to itself. Moreover, since the symbol $\fM(\ell,v(j),b(j))$ of a line-T\"oplitz operator has no conditions on $\ell$, they are also well-behaved with respect to the line-T\"oplitz structure.  
By direct inspection one then gets the following Lemmata.
\begin{lemma}\label{proiettotutto}
The projections $\Pi_{|\ell|\le K}$ (and $\Pi_{|\ell| > K}$) preserve the spaces $\cT_{a,m}^\cO$, $\op{\cO}{a}$, and they are continuous. In particular, if $M$ is quasi-T\"oplitz then so is its time average
$\jap{M}_{\T^d}:= \Pi_{\ell=0} M.$.\\ For all $a'< a$, one has the bounds:
\begin{gather*}
|\Pi_{|\ell|\le K} M|^{\mathtt{ T},\cO}_{a,-m} \le |M|^{\mathtt{ T},\cO}_{a,-m}, \quad
\segm{\Pi_{|\ell|\le K} M}{\cO}{a} \le \segm{M}{\cO}{a}\,,\\
|\Pi_{|\ell|> K} M|^{\mathtt{ T},\cO}_{a',-m} \le e^{-(a-a')K} |M|^{\mathtt{ T},\cO}_{a,-m}, \qquad
\segm{\Pi_{|\ell|> K} M}{\cO}{a'} \le e^{-\frac{a-a'}{2}K} \segm{M}{\cO}{a}.
\end{gather*}
%
\end{lemma}

	\begin{lemma}\label{diago}
	Given a \emph{time independent} (and hence diagonal) operator $A= \diag{(A_j)}\in \op{\cO}a$,  
	there exists a decomposition of the eigenvalues 
	\[
	A_j = \mathfrak a (v(j),b(j)) + \sum_{k=1}^m r^{(k)}_j
	\]
	so that 
	\begin{equation}\label{stima diago}
		\sup_{v\in \cV,b\in \Z} |\fa(v,b)|^{\cO} e^{a{\tc}|v|} {\jap{b}^m}+ \sum_{k=1}^m\sup_j |r^{(k)}_j |^{\cO}\jap{j}^{k\mu} \jap{b(j)}^{m-k} \le 2 \segm{A}{\cO}{a}
	\end{equation}
	Moreover, if $A_j\in \R$, then also $\fa,r^{(k)}_j\in\R$ for $k = 1, \ldots m$.
	\end{lemma}
\begin{proof}
	Since $A\in \cL_{a,-m}^{\tt qT, \cO}$ there exists a decomposition
	\[
	A = A^\tT +  \sum_{k=1}^m A^{(k)}\,, 
	\quad \mbox{	such that}\quad
	|A^\tT|^{\mathtt {T},\cO}_{a,-m} +   \sum_{i = 1}^{m}|A^{(i)}|^\cO_{\frac{a}{2}; -(i, m-i)}\le 2 \segm{A}\cO{a}.
	\]
Recalling that $A$ is time-independent, we have
\[
A = \jap{ A}_{\T^d} = \jap{A^\tT}_{\T^d} +  \sum_{k=1}^m \jap{A^{(k)}}_{\T^d}.
\]
Now, the operators $ \jap{A^\tT}_{\T^d}, \jap{A^{(k)}}_{\T^d}$ are all time-independent and hence diagonal; moreover, $ \jap{A^\tT}_{\T^d}$ is still line T\"oplitz. The bounds follow from Remark \ref{rem:maj.norm}.  The reality condition follows by taking the real part in both sides of the equality above.
\end{proof}

{We now discuss the algebra properties of $\op{\cO}{a}$. Our purpose is to get a result analogous to the one in Lemma \ref{lemma.algebra}. With this in mind, we give two propositions on the product of operators in $\op{\cO}{a}$.}
\begin{proposition}\label{prop.chiavetta}
	Let $M_1\in \cL^{\mathtt{qT}, \cO}_{a, -m_1}$ and $M_2\in\cL^{\mathtt{qT}, \cO}_{a,-m_2}$, and set
	$m = \min\{m_1, m_2\}$.
	Then we have $M_1M_2\in\cL^{\mathtt{qT}, \cO}_{a,-m}$, with the bound
	\begin{equation}\label{stima.chiavetta}
	\segm{M_1 M_2}{\cO}{a} \leq C a^{-q_0}
	\abs{M_1}^{\mathtt{qT}, \cO}_{a,-m_1} \, 
	\abs{M_2}^{\mathtt{qT}, \cO}_{a,-m_2}\,.
	\end{equation}
	Moreover, setting $m_1 = m_2 = 1$, one also has $\forall 0<a'< a$ that $M_1 M_2 \in  \cL^{\mathtt{qT}, \cO}_{a',-2}$, with
	\begin{equation}\label{a.mano}
		\sedue{M_1 M_2}{\cO}{a'} \leq {C(a - a')^{-q_1}}
		\seuno{M_1}{\cO}{a} \, 
		\seuno{M_2}{\cO}{a}\,.
	\end{equation}
	Here $C, q_0, q_1$ are positive constants depending only on $m_1, m_2, \mu, \delta, \tc, p$.
\end{proposition}
The proof, being quite technical, is postponed to Appendix \ref{sec:chiavetta}.

\begin{corollary}[Exponentials]\label{esponenziale}
	Given $a,m>0$ and  $A\in  \cL^{\mathtt{qT}, \cO}_{a,-m}$ such that (the constants $C,\gamma$ are defined in Proposition \ref{prop.chiavetta})
	\begin{equation}\label{piccolezza}
	\segm{A}{\cO}{a}  \le \delta:= \frac{a^{q_0}}{4 C}
	\end{equation}
	one has for any sequence $c_k\in \ell_\infty$, and for all $M\in  \cL^{\mathtt{qT}, \cO}_{a,-m}$
	\begin{equation}\label{stimaexp}
	\segm{\sum_{k=d}^\infty c_k \ad(A)^k M}{\cO}{a} \le 2 |c|_\infty  \Big({\frac{\segm{A}{\cO}{a}}{2\delta}}\Big)^d 
	\end{equation}
\end{corollary}
\begin{proof}
	It follows directly by iterating the bounds in  Proposition \ref{prop.chiavetta}.
\end{proof}


{
\begin{remark}
 Estimate \eqref{stima.chiavetta} of Proposition \ref{prop.chiavetta} and Corollary \ref{esponenziale} ensure that the well known fact that smoothing operators generate bounded changes of variables holds also in the quasi-T\"oplitz context. On the other hand, estimate \eqref{a.mano} shows that the product of two operators of order $-1$ is actually of order $-2$, up to a small loss of analyticity.
\end{remark}}
\section{The homological equation}
\label{sec:hom.eq}
We consider a  diagonal operator 
\begin{equation}
	\label{diagonale}
	D =   \diag(\Omega_j)_{j \in \Z^2} , \quad 
	\Omega_j = |j|^2 + \widetilde \Omega_j , \quad  \widetilde \Omega:= \diag ( \widetilde \Omega_j)  \in \opdue{\tR}{a}  \ . 
\end{equation}
\vskip-5pt
Recalling Lemma \ref{diago}, we have {$\widetilde\Omega_j=  \mathfrak a(v(j),b(j))+ r^{(1)}_j +  r^{(2)}_j$}, with $\mathfrak a(v,b)$, $r^{(1)}$, $r^{(2)}$ satisfying \eqref{stima diago} for $m=2$.
		For all  $0<|\ell|\le K$,  all $j\in \Z^2$  and all $v\in \cV$ such that $\pi(\ell)\parallel v$ 
		we define  \begin{align}\label{lavello}
			\td(\ell,j)&:= \omega \cdot \ell +\Omega_j -\Omega_{j-\pi(\ell)}
			\\
			\label{d.decomp}
			\mathfrak d(\ell,v,b)&:= \omega \cdot \ell +2 \frac{|\pi(\ell)|}{|v|} b -|\pi(\ell)|^2+ \fa(v,b) - \fa(v,  b- v\cdot \pi(\ell)) 
		\end{align}
		\vskip-5pt
		We define  the (possibly empty) sets  for $\gamma \geq 0$, $K \in \N$
		\begin{gather}
			\label{CD1}
			\cC^{(1)}\equiv \cC^{(1)}_{D,K}(\gamma):= \left\{ \omega\in \cO:\ \  
			| \td(\ell, j) | \ge \g |\ell|^{-\tau}\,,\;\forall\, 0<|\ell|\le K, \ \ \ j\in \Z^2 \right\}
		\\
			\label{CD2}
			\cC^{(2)}\equiv \cC^{(2)}_{D,K}(\gamma):= 
			\left\{ \omega\in \cO:\quad|\mathfrak d(\ell,v,b)|\ge 2\g |\ell|^{-\tau} \,,\;\forall\, 0<|\ell|\le K, \ \ v\in \cV \mbox{ s.t.} \; \pi(\ell)\parallel v\right\},\\
			\label{OO}
			\cO_{D,K}(\g):= \cC^{(1)}_{D,K}(\gamma)\cap \cC^{(2)}_{D,K}(\gamma).
		\end{gather}
		\begin{proposition}\label{homolog}
			Fix $\gamma,a > 0$, $K \in \N$ and a compact set $\cO\subseteq \cO_0\subset\tR$.
			Consider a diagonal operator  $D$ as in \eqref{diagonale} with 
			$16 \sedue{\widetilde\Omega}{\tR}{a} \le \g\,$.
			Consider the sets $\cC^{(1)}$, $\cC^{(2)}$ defined in \eqref{CD1}, \eqref{CD2} and set
			\begin{equation}
				\label{Opiu}
				\cO_{+}= \cO_{D,K}(\g) \ . 
			\end{equation}
			For all $P \in \opdue{\cO}{a}$, 
			there exists  $S \in \opdue{\cO_+}{a}$   solving the  homological equation 
			\begin{equation}
				\label{hom.eq}
				-\im  \dot  S +	[D,S] = \Pi_{0<|\ell|\le K} P,
			\end{equation}
			fulfilling the estimate
			\begin{equation}
				\label{stimahomol}
				\sedue{S}{\cO_+}{a}\le C(p, \mu, \delta, \tc)\, a^{-2/\delta} \, \g^{-1} \,  K^{3\tau+1}\, \sedue{P}{\cO}{a}.
			\end{equation}
			If $P$ is self-adjoint, then so is $\im S$; {if $P$ is Gauge covariant, then so is $S$}.
		\end{proposition}
		{The solution $S$ of the homological equation \eqref{hom.eq} is given, in components, by 
			\begin{equation}\label{hom.eq.coord}
				S_j^{j-\pi(\ell)}(\ell)= 
				\dfrac{P_j^{j-\pi(\ell)}(\ell)}{ \omega \cdot \ell +\Omega_j -\Omega_{j-\pi(\ell)}} 
				\quad  \mbox{if}\; 0<|\ell|\le K, \qquad 
				S_j^{j-\pi(\ell)}(\ell)= 0 \quad \mbox{otherwise}.
			\end{equation}
			Note that $S$ is well defined for $\omega\in \cC^{(1)}_D$ and it is well known that $S\in \cL_a(\fH^p)$. 
			Moreover if $P$ is self-adjoint, one verifies that $\im S$ is so (just use the characterization  \eqref{AA}), {same for the Gauge covariance.}
			\\
			Next we  show that
			$S\in \opdue{\cO_+}a$, namely it has  a decomposition   (cfr.  \eqref{arco}) 
			\[
			S= S^\tT + S^{(1)}+ S^{(2)} \mbox{ with }S^\tT \in \cT^{\cO_+}_{a,-2}, \ \ S^{(1)} \in \cL^{\cO_+}_{\frac{a}{2}; - (1,1)}, \ \ 
			S^{(2)} \in \cL^{\cO_+}_{\frac{a}{2}; - (2,0)} \ . 
			\]
			By assumption also  $P$ is admissible of order $-2$, hence it has the same decomposition. 
			In particular we shall denote the line-T\"oplitz part of $P$ by 
			\begin{equation}
				\label{P.lT}
				[P^\tT]_{j}^{j-\pi(\ell)} = \mathfrak P(\ell, v(j), b(j)) \ . 
			\end{equation} 
			We then decompose $S$ as follows
			\begin{align}
				\label{S.decomp}
				S_j^{j-\pi(\ell)}(\ell) = 
				\frac{[P^\tT]_j^{j-\pi(\ell)}(\ell)}{ \omega \cdot \ell +\Omega_j -\Omega_{j-\pi(\ell)}} 
				+
				\underbrace{\frac{[P^{(1)}]_j^{j-\pi(\ell)}(\ell)}{ \omega \cdot \ell +\Omega_j -\Omega_{j-\pi(\ell)}}}_{=: [\widetilde S^{(1)}]_{j}^{j-\pi(\ell)}}
				+
				\underbrace{\frac{[P^{(2)}]_j^{j-\pi(\ell)}(\ell)}{ \omega \cdot \ell +\Omega_j -\Omega_{j-\pi(\ell)}}}_{=: [\widetilde S^{(2)}]_{j}^{j-\pi(\ell)}} \ . 
			\end{align} 
			We further need to decompose the first term in the r.h.s. above, which is 
			the most delicate one because of the divisor.
			We introduce the following sets
			\begin{equation}
				\label{}
				\begin{aligned}
					& A_0 := \{ (\ell, j) \in \Z^d \setminus\{0\} \times \Z^2 \colon \ \ \pi(\ell) = 0 , \quad |\ell| < K \} \\
					& A_1:= \{(\ell, j) \in \Z^d \setminus\{0\} \times \Z^2 \colon \ \ \pi(\ell) \neq  0 , \quad |\ell| < K  ,  \ \  v(j) \parallel \pi(\ell)   \}\\
					& A_2:= \{(\ell, j) \in \Z^d \setminus\{0\} \times \Z^2 \colon \ \ \pi(\ell) \neq  0 , \quad |\ell| < K  ,  \ \  v(j) \not\parallel \pi(\ell), \ \  |\ell| < \tc \la j \ra^\delta \}  \\
					& A_3:= \{(\ell, j) \in \Z^d \setminus\{0\} \times \Z^2 \colon \ \ \pi(\ell) \neq  0 , \ \  v(j) \not\parallel \pi(\ell), \quad  \tc \la j \ra^\delta \leq |\ell| < K  \} 
				\end{aligned}
			\end{equation} 
			and note that $\{ (\ell, j) \colon 0 < |\ell| < K\} = \cup_{i=0}^3 A_i$. 
			Further denote by $\delta_A:=\delta_A(\ell,j)=1$ if $(\ell,j)\in A$, $\delta_A(\ell,j)=0$ otherwise.
			Starting from \eqref{S.decomp} we further decompose 
			\begin{align}
				\label{S.dec1}
				S_j^{j-\pi(\ell)}(\ell) 
				& = 
				\frac{[P^\tT]_j^{j-\pi(\ell)}(\ell)}{\td(\ell,j)}  \delta_{A_0}
				+ 
				\frac{[P^\tT]_j^{j-\pi(\ell)}(\ell)}{ \mathfrak d(\ell,v(j),b(j)) }  \delta_{A_1}
				\\
				\label{S.dec2}
				& 
				+ 
				[P^\tT]_j^{j-\pi(\ell)}(\ell)\left( \frac{1}{\td(\ell,j)} - \frac{1}{ \mathfrak d(\ell,v(j),b(j)) } \right)   \delta_{A_1}
				+
				\frac{[P^\tT]_j^{j-\pi(\ell)}(\ell)}{\td(\ell,j)}  \delta_{A_2}
				+
				[\widetilde S^{(1)}]_{j}^{j-\pi(\ell)}
				\\[-7pt]
				\label{S.dec3}
				& +
				\frac{[P^{\tT}]_j^{j-\pi(\ell)}(\ell)}{ \td(\ell,j)} \delta_{A_3}+
				[\widetilde S^{(2)}]_{j}^{j-\pi(\ell)}.
			\end{align}
			In the next lemmata we  shall  prove the line  \eqref{S.dec1} gives a line-T\"oplitz operator, 
			line \eqref{S.dec2} gives an operator in $\cL^{\cO_+}_{\frac{a}{2}; - (1,1)}$ and line 
			\eqref{S.dec3} one in $\cL^{\cO_+}_{\frac{a}{2}; - (2,0)}$.
			\begin{lemma}\label{lem.hom1}
				The operator in the r.h.s.\ of line \eqref{S.dec1}, which we denote by $S^\tT$,   is line-T\"oplitz and 
				\begin{equation}
					\label{Stt.est}
					|S^\tT|^{\tT,\cO_+}_{a,-2}\le  \frac{1}{2 \gamma}K^{{2\tau+1}}|P^\tT|^{\tT,\cO}_{a,-2}.
				\end{equation}
			\end{lemma}
			\begin{proof}
				Recalling \eqref{P.lT}, \eqref{d.decomp}, one checks that  the elements of $S^\tT$ are actually given by 
				$(S^\tT)_j^{j-\pi(\ell)}(\ell):= \fS(\ell,v(j),b(j))$ with 
				\vspace{-10pt}
				\[
				\fS(\ell,v,b):= \begin{cases}
					\dfrac{ \mathfrak P(\ell,v,b)}{ \omega\cdot \ell } \quad &{\rm if} \quad \pi(\ell)= 0 \,,\quad \ell\ne 0\,,\;|\ell|<K\\ 
					\dfrac{ \mathfrak P(\ell,v,b)}{\mathfrak d(\ell,v,b)}
					\qquad &{\rm if} \quad \pi(\ell)\ne 0\,,\; |\ell| < K\,,\;  v\parallel \pi(\ell)\\ 
					0 \qquad & {\rm otherwise}
				\end{cases}
				\]
				so $S^\tT$ is line-T\"oplitz. 
				To study the Lipschitz variation we remark that,
				since
				\[
				\abs{\Delta_{\omega,\omega'}\mathfrak d(\ell,v,b) } \le |\ell| +2 \sedue{\widetilde\Omega}{\cO}a \le 2K\qquad\forall\omega\ne\omega'\in \cO_+,
				\]
				\\[-20pt]
				one has
				\[
				\left|\Delta_{\omega,\omega'} \frac{1}{\mathfrak d(\ell,v,b)}\right|= \left|\frac{\Delta_{\omega,\omega'} \mathfrak d(\ell,v,b)}{\mathfrak d(\ell,v,b)(\omega)\, \mathfrak d(\ell,v,b)(\omega')}\right|\le 2 \g^{-2}K^{2\tau+1}\,.
				\]
				It follows that
				$
				|\Delta_{\omega,\omega'} S^\tT|^\tT_{a,-2}\le \g^{-1} K^{\tau}|\Delta_{\omega,\omega'} P^\tT|^\tT_{a,-2}
				+
				2 \g^{-2}K^{2\tau+1}|P^\tT|^\tT_{a,-2}.
				\,
				$
				Estimate \eqref{Stt.est} follows using the small divisor estimates in \eqref{CD1}, \eqref{CD2}.
			\end{proof}
			\begin{lemma}\label{lem.hom2}
				The operator in line \eqref{S.dec2}, which we denote by $S^{(1)}$, belongs to $\cL^{\cO_+}_{\frac{a}{2}; -(1,1)}$ and there exists $C = C(p, \mu, \delta, \tc)>0$ such that 
				\begin{equation}
					\label{stimaS} 
					|S^{(1)}|_{a/2;-(1,1)}^{\cO_+} \le  C \,  \g^{-1} K^{3\tau+1} 
			\, 	 \sedue{P}{\cO}{a} \,.
				\end{equation}
			\end{lemma} 
			\begin{proof}
				We bound first the operator $\widetilde S^{(1)}$, defined in \eqref{S.decomp}.
				Using  $\omega\in \cC^{(1)}_D$, we know that 
				$|(\wtS^{(1)})_j^{j-\pi(\ell)}(\ell)|\le \g^{-1}K^\tau |(P^{(1)})_j^{j-\pi(\ell)}(\ell)|$. 
				Thus, by Remark \ref{rem:maj.norm}, we get
				$
				\nor{\wtS^{(1)}}{a/2}{-(1,1)} \le \g^{-1} K^{\tau}\nor{P^{(1)}}{a/2}{-(1,1)}
				$.
				To study the Lipschitz norm we  proceed as in the previous Lemma 
				and  we conclude that
				\begin{equation}
					\label{S1.11}
					|\wtS^{(1)}|_{a/2;-(1,1)}^{\cO_+} \le  2 \g^{-1} K^{2\tau+1}|P^{(1)}|_{a/2;-(1,1)}^{\cO}\,.
				\end{equation}
				Regarding the remaining operator in line \eqref{S.dec2}, which we denote by $\widehat S^{(1)}$, it has matrix elements given explicitly by
				\vspace{-10pt}
				\[
				(\widehat S^{(1)})_j^{j-\pi(\ell)}(\ell):= 
				\begin{cases}
					\dfrac{(P^\tT)_j^{j-\pi(\ell)}(\ell)}{\td(\ell, j)}\,,\qquad \qquad\qquad \mbox{if} \quad\pi(\ell)\ne 0\,,\; &|\ell|< \min(K, \tc \jap{j}^\delta)\,,\;  v(j)\not\parallel \pi(\ell)\\
					(P^\tT)_j^{j-\pi(\ell)}(\ell)(\dfrac{1}{\td(\ell,j)} - \dfrac{1}{\mathfrak d(\ell,v(j),b(j))})\,,\quad & \mbox{if}\quad \pi(\ell)\ne 0\,,\; |\ell|<K \,,\;  v(j)\parallel \pi(\ell)\\
					0 \qquad & {\rm otherwise.}
				\end{cases}
				\]
				{\em First line of $\widehat S^{(1)}$:} 
				We claim  that there exists $C=C(\delta,\mu,\g)$ such that
				\begin{equation}
					\label{S1.est.1}
					\sup_{\substack{j\in \Z^2, \ell\in \Z^d:\\ |\ell|<\tc \jap{j}^\delta\,, \pi(\ell)\not\parallel v(j)}}
					\left|{\frac{\jap{j}^\mu}{\omega \cdot \ell +\Omega_j -\Omega_{j-\pi(\ell)}} }\right|^{\cO_+}\le C \g^{-1} K^{2\tau+1} \,.
				\end{equation}
				Indeed  if $\jap{j}\le \tJ_\delta$ (with $\tJ_\delta$ the constant of Lemma \ref{traviata}), then the bound follows trivially using also that $\omega \in \cC^{(1)}$. \\
				Consider now the case $\jap{j}> \tJ_\delta$. 
				Then by \eqref{d.decomp}, in order for the denominator to gain the term $\la j \ra^\mu$, we need $\pi(\ell)\cdot j$ to be large.
				This is what we show to happen. In particular 
				we claim that 
				\begin{equation}
					\label{S1.claim}
					\jap{j}\ge \tJ_\delta , \ \ |\ell| < \tc \jap{j}^\delta, 
					\ \ 
					\pi(\ell)\not\parallel v(j)
					\quad \Rightarrow \quad  |j\cdot \pi(\ell)|> 2 \jap{j}^{\mu}
				\end{equation}
				and then \eqref{S1.est.1} still holds.
				To prove \eqref{S1.claim} we distinguish 2 cases.\\
				\underline{Case 1: $|b(j)|> 2\jap{j}^{\mu}$.}
				Let $w\in \cV$  be the  direction parallel to $\pi(\ell)$, then (being $w$ a generator), 
				$|w| \leq |\pi(\ell)| \leq \tc^{-1}|\ell| \leq \la j \ra^\delta$.
				As  $v(j)$ is the vector realizing  $\min \{| v\cdot j| \colon v \in \cV, \ |v| \leq |j|^\delta\}$, we have
				\[
				|\pi(\ell)\cdot j|\ge |w\cdot j|\ge |v(j)\cdot j| = |b(j)| > 2\jap{j}^{\mu}
				\]
				and in this case \eqref{S1.claim} follows.\\
				\underline{Case 2: $|b(j)|\leq 2 \jap{j}^{\mu} $.} 
				Then as $|w| \leq \la j \ra^\delta$ and $w \neq v(j)$, Lemma \ref{traviata} $(i)$ gives again $|w \cdot j | \geq 2 \la j \ra^\mu$,
				proving \eqref{S1.claim} also in this case. \\ 
				Now, recalling that 
				\begin{align}
					\td(\ell,j)
					= \omega \cdot \ell  +2\pi(\ell)\cdot j -|\pi(\ell)|^2  + \widetilde\Omega_j -\widetilde\Omega_{j-\pi(\ell)}\,,
				\end{align}
				and using \eqref{S1.claim} we obtain that 
				$
				|\td(\ell,j) |\ge 2\jap{j}^{\mu} - (|\omega| \tc +1 ) \jap{j}^{2\delta}-1 \geq C \la j \ra^\mu
				$.
				Hence  \eqref{S1.est.1} is proved in case of the sup norm.  The Lipschitz norm is estimated by the same case analysis.
				\\
				We deduce that
				\begin{equation}
					\label{S1.12}
					\nor{\widehat S^{(1)} \delta_{A_2}}{a/2}{-(1,1)}^{\cO_+} \le  C \g^{-1}K^{2\tau+1}  \nor{P^\tT}{a/2}{-(0,2)}^{\cO}
\le		\left(\frac{a}{2}\right)^{-p-2}C \g^{-1}K^{2\tau+1}  |P^\tT|^{\tT,\cO}_{a,-2}
					\, , 
				\end{equation}
				where in the last inequality we used Lemma \ref{vaccab}.\\
				{\em Second line of $\widehat S^{(1)}$:} Again we consider two cases.		\\
		\underline{Case 1: $|b(j)|> 2\jap{j}^{\mu}$ or $\langle v(j)\rangle > \frac12 \langle j \rangle^\delta$ or $|\ell|> \tc \langle j \rangle^\delta$.} 		In this case we exploit 
		the decay of $(P^\tT)_j^{j-\pi(\ell)}$ and get that for $\omega \in \cO_{D,K}(\gamma)$
		\begin{align}
		\notag
	|b(j)|  \, 	\langle j \rangle^{\mu} \, & \abs{	(P^\tT)_j^{j-\pi(\ell)}(\ell)(\dfrac{1}{\td(\ell,j)} - \dfrac{1}{\mathfrak d(\ell,v(j),b(j))})}\\
	\notag
		&\leq 
		|b(j)| \,	\langle j \rangle^{\mu} \, 
			\abs{	(P^\tT)_j^{j-\pi(\ell)}(\ell)} \,  \left( \abs{\dfrac{1}{\td(\ell,j)}} + \abs{ \dfrac{1}{\mathfrak d(\ell,v(j),b(j))})}	\right)\\
			\label{case1.est}
			& \leq C \, a^{-\frac{\mu}{\delta}}
		\gamma^{-1} K^{\tau} \  e^{- \frac{a \tc }{2}|v(j)| - \frac{3a }{4}  |\ell|} \  |P^\tT|^{\tT,\cO}_{a,-2} \ . 
		\end{align}
	\underline{Case 2: $|b(j)|\leq  2\jap{j}^{\mu}$ and  $\langle v(j)\rangle \leq \frac12 \langle j \rangle^\delta$ and  $|\ell|\leq \tc \langle j \rangle^\delta$.} 		
	In this case we exploit the decay of $\td(\ell,j)^{-1} - \mathfrak d(\ell,v(j),b(j))^{-1}$. 
	Precisely we note that 
\begin{align*}
\td(\ell,j) & := \omega \cdot \ell + |j|^2 -|j-\pi(\ell)|^2 + \widetilde\Omega_j -\widetilde\Omega_{j-\pi(\ell)}
				= \omega \cdot \ell  +2\pi(\ell)\cdot j -|\pi(\ell)|^2  + \widetilde\Omega_j -\widetilde\Omega_{j-\pi(\ell)} \ , 
			\end{align*} 
			so decomposing $\widetilde \Omega$ as in Lemma \ref{diago}, 
			and using  that $v(j)   \parallel \pi(\ell) $
			  we arrive at
			\begin{align*}
				\td(\ell,j) 
				=&  \omega \cdot \ell  +2\frac{|\pi(\ell)|}{|v(j)|}b(j) -|\pi(\ell)|^2
				+ \ta(v(j), b(j)) - \ta(v(j-\pi(\ell)), b(j-\pi(\ell))) \\
&				+ r_j^{(1)}
				-r_{j-\pi(\ell)}^{(1)} + 
				r_j^{(2)}	
				- r_{j-\pi(\ell)}^{(2)} \ .
			\end{align*}
Since $|b(j)|\leq  2\jap{j}^{\mu}$,  $\langle v(j)\rangle \leq \frac12 \langle j \rangle^\delta$ and also  $|\pi(\ell)| \leq \langle j \rangle^\delta$, Lemma 
	\ref{traviata} $(ii)$ implies that $v(j) = v(j-\pi(\ell)) $ and then 
	one has also $b(j-\pi(\ell)) = b(j) - v(j)\cdot\pi(\ell)$.
			We deduce  that
			$$
			\td(\ell,j)  = \mathfrak d(\ell,v(j),b(j)) + r_j^{(1)}
				-r_{j-\pi(\ell)}^{(1)} + 
				r_j^{(2)}	
				- r_{j-\pi(\ell)}^{(2)} \ 
			$$
			and thus we estimate 
\begin{align*}
			\abs{\jap{j}^\mu(\dfrac{1}{\td(\ell,j)} - \dfrac{1}{\mathfrak d(\ell,v(j),b(j))})}& = \abs{\dfrac{\jap{j}^\mu ( r_j^{(1)}	-r_{j-\pi(\ell)}^{(1)} + 	r_j^{(2)}	 - r_{j-\pi(\ell)}^{(2)})}{\td(\ell,j) \mathfrak d(\ell,v(j),b(j))}} \\
			 &				\le \sedue{\tilde\Omega}{\cO_0}a \g^{-2} K^{{2\tau}} {\le}  \g^{-1} K^{2\tau}
\end{align*}
				We reason in the same way for the Lipschitz variation and obtain the bound
				\begin{equation}
					\label{S1.est.2}
					\sup_{\substack{j\in \Z^2, \ell\in \Z^d:\\ |\ell| < K, \ \  \pi(\ell)\parallel v(j)}}
					\left|	\jap{j}^\mu \left(\dfrac{1}{\td(\ell,j)} - \dfrac{1}{\mathfrak d(\ell,v(j),b(j))}\right) \right|^{\cO_+}\le C \g^{-1} K^{3\tau+1} \,.
				\end{equation}
				Thus we obtain the same kind of estimate as in \eqref{case1.est}, with $\tau \leadsto 3\tau+1$.\\
	In both cases, using Lemma  \ref{gigetto}, we finally 	get
				\begin{equation}
					\label{S1.13}
					\nor{\widehat S^{(1)} \delta_{A_1}}{a/2}{-(1,1)}^{\cO_+} \le  C \g^{-1}K^{3\tau+1}   |P^\tT|^{\tT,\cO}_{a,-2} \,.
				\end{equation}
				Then estimate \eqref{stimaS} follows from \eqref{S1.11}, \eqref{S1.12}, \eqref{S1.13}.
			\end{proof}
			\begin{lemma}\label{lem.hom3}
				The operator in line \eqref{S.dec3}, which we denote by $S^{(2)}$, belongs to $\cL_{{\frac{a}{2}; -(2,0)}}^{\cO_+}$ and 
				there exists $C = C(p,\mu,\delta,\tc) >0$ such that 
				\begin{equation}
					\label{stimaS2} 
					|S^{(2)}|_{a/2;-(2,0)}^{\cO_+} \le C \g^{-1} K^{2\tau+1}\left( |P^{(2)}|_{a/2;-(2,0)}^{\cO}  + a^{-2{\mu}/\delta}  |P^\tT|^{\cO}_{a}
					\right)
				\end{equation}
			\end{lemma}
			\begin{proof}
				We need to estimate two contributions. The operator $\widetilde S^{(2)}$, given in \eqref{S.decomp}, is estimated similarly as $\widetilde S^{(1)}$ in the previous lemma, obtaining the estimate
				\begin{equation}
					\label{}
					\nor{\wtS^{(2)}}{a/2}{-(2,0)}^{\cO_+} \le \g^{-1} K^{2\tau+1}\nor{P^{(2)}}{a/2}{-(2,0)}^{\cO}\,.
				\end{equation}
				Regarding the first operator in \eqref{S.dec3}, which we denote by $\widehat S^{(2)}$, it has elements only supported in  the region $|\ell| \geq \tc \la j \ra^\delta$. 
				Then  the small divisor estimate in \eqref{CD1} and Lemma \ref{Bony} give the bound
				\[
				|\widehat S^{(2)}|^{\cO_+}_{a/2, -(2,0)}  \le  C(p,\mu,\delta,\tc) \, a^{-2{\mu}/\delta}  |\widehat S^{(2)}|^{\cO_+}_{a}\le \g^{-1}C(p,\delta) \, a^{-2{\mu}/\delta} K^{2\tau+1}
				|P^\tT|^{\cO}_{a}\,.
				\qedhere
				\]
		\end{proof}}
		\begin{proof}[Proof of Proposition \ref{homolog}]	Estimate \eqref{stimahomol} follows from Lemmata \ref{lem.hom1}, \ref{lem.hom2} and \ref{lem.hom3}.
		\end{proof}
		\section{The KAM reduction scheme}
		In this section we shall study the linear Schr\"odinger operator in \eqref{Lo} and prove that, provided that $V$ in \eqref{travel} is sufficiently small, then $\fL$ is reducible with sufficient knowledge on the asymptotics of the eigenvalues. As explained in the introduction, the reducibility of Schr\"odinger operators is by now well-known, the novelties in this result are the fact that the reducing change of variables is quasi-T\"oplitz of order $-1$ and the asymptotics of the eigenvalues.
		
		\begin{theorem}[Reducibility] \label{riduco}
			Consider a Schr\"odinger operator as in \eqref{Lo} with a travelling wave potential $\VV$ as in \eqref{travel}.
			There exist $\epsilon_*,\gamma_*>0$ such that the following holds. For all $\gamma \in (0,\gamma_*)$ and for any potential satisfying  the smallness condition \eqref{pic colo0}
			there exist: 
			\\
			 a sequence of Lipschitz functions $\mathtt R:= [-1,1]^d\ni \omega \mapsto \Omega_j(\omega)$, 
			a positive measure set $\cO_{\infty}\subset\tR$ and a sympletic change of variables $G$ such that 
			\begin{equation}
				\label{cerca}
				G_* \fL= \diag(\Omega_j)\,,\quad 
				G-\id \in\opuno{\cO_{\infty}}{\ta'} \,,\quad \mbox{with}\quad {\ta'= \ta/8}\,.
			\end{equation}
			Moreover 
			\begin{equation}
				\label{Omegaj}
				\Omega_j= |j|^2 + \fa(v(j),b(j)) + r^{(1)}_j + r^{(2)}_j 
			\end{equation}
			with the bounds
			\[
			\sup_{v\in \cV,b\in \Z} |\fa(v,b)|^{\tR} e^{\ta' {\tc}|v|} {\jap{b}^2}+\sup_j |r^{(1)}_j |^{\tR}\jap{j}^{\mu} \jap{b(j)} 
			+ \sup_j |r^{(2)}_j |^{\tR}\jap{j}^{2\mu} \le 2 \|V\|_{\ta,\tp}\,.
			\]
			{Finally  if $V$ is Gauge invariant then $G$ is Gauge covariant.}
		\end{theorem}
		We can be more explicit regarding the set $\cO_{\infty}$.
		For $\ell\in \Z^d$, $j\in \Z^2$  and all $v\in \cV$ such that $\pi(\ell)\parallel v$ we  consider $\td(\ell,j),\mathfrak d(\ell,v,b)$ defined as in \eqref{lavello} with $\Omega_j$ defined in \eqref{Omegaj}.
		
		{\begin{proposition}\label{prop:Ofin} Under the Hypotheses of Theorem \ref{riduco}, there exists $\tau>0$ such that 	the set $\cO_{\infty}$ contains the set 
				$
				\cO_{\rm fin}:=\cO_0(\g)\cap 	\cC^{(1)}(\g,\tau)\cap \cC^{(2)}(\g,\tau)
				$
				where
				\begin{gather}
					\label{diofantino}
					\cO_0=\cO_0(\g):= \{\omega\in \tR:\quad |\omega\cdot\ell + k| \ge  2\g|\ell|^{{-(d+1)}}\quad \forall (\ell,k )\in \Z^{d+1}:\; \ell\ne 0 \}\,,
			\\
					\label{CD1fin}
					\cC^{(1)}= \cC^{(1)}(\g,\tau):= \left\{ \omega\in \tR:\ \  
					| \td(\ell, j) | \ge 2\g |\ell|^{-\tau}\,,\;\forall\, \ell\ne 0, \ \ \ j\in \Z^2 \right\}\,,
			\\
					\label{CD2fin}
					\cC^{(2)}:= \cC^{(2)}(\g,\tau) := 
					\left\{ \omega\in \tR:\quad|\mathfrak d(\ell,v,b)|\ge 4\g |\ell|^{-\tau} \,,\;\forall\, \ell\ne 0, \ \ v\in \cV \mbox{ s.t.} \; \pi(\ell)\parallel v\right\} \,.
				\end{gather}
		\end{proposition}}
		We defer the proof of this proposition to the end of the section.
		
		\paragraph{The first step.}
		In the first step in the reduction scheme we exploit the structure of the Schr\"odinger operator in \eqref{Lo} and perform a change of variables $G_0= e^{S_0}$ which conjugates $\fL$ to $-\Delta +P_0$,
		with $P_0$ quasi-T\"oplitz of order $-2$. 
		\begin{lemma}\label{passozero}
			Fix {$a_0=\ta/4$} ,  $\g>0$ and consider a  Schr\"odinger operator as in \eqref{Lo} satisfying \eqref{pic colo0}. Then there exist positive constants $C, q_2, q_4$, depending on $\delta, \mu, \tau, \tc, p$ only, and operators $S_0\in \opuno{\cO_0}{\frac\ta3}$ and $P_0\in \opdue{\cO_0}{a_0}$ (with $\cO_0$ defined in \eqref{diofantino})
			satisfying the bounds
			\[
			\seuno{S_0}{\cO_0}{\frac\ta3}\le  C\g^{-1} \ta^{-q_2} \|V\|_{\ta,\tp}\,,\quad 	\sedue{P_0}{\cO_0}{a_0}\le  C \g^{-1}  \ta^{-{q_4}} \|V\|_{\ta,\tp}^2\,,
			\]
			such that
			$
			\fL_0:=	(G_0)_* \fL = -\Delta + P_0
			$.
		\end{lemma}
		The proof is performed in several steps.
		We start by defining $S_0$ as the solution of the homological equation
		(see \eqref{def:Mv})
		\vspace{-10pt}
		\begin{equation}\label{hom.sweet.hom}
			\im \,\dot S_0 +	[\Delta,S_0] =  M_V\;\iff \; (S_0)_j^{j-\pi(\ell)}(\ell)= \begin{cases}
				- \dfrac{V(\ell)}{ \omega\cdot\ell +|j|^2 -|j-\pi(\ell)|^2} 
				\quad & \mbox{if}\; \pi(\ell)\ne 0\\
				0 & \mbox{otherwise}
			\end{cases} 
		\end{equation}	
		(recall that, by \eqref{condV}, $V(\ell)=0$ if $\pi(\ell)=0$). 
		\\
		Note that $S_0$ is well defined for $\omega\in \cO_0$ and
		\begin{align}
			\label{Gstar0}
			(G_0)_* (-\Delta+M_V)&= e^{\ad S_0 } (-\Delta+M_V)-\im \sum_{h= 1}^\infty \frac{(\ad S_0)^{h-1}}{h!}\dot S_0  
			\\&=-\Delta+M_V + \sum_{h= 1}^\infty \frac{(\ad S_0)^{h-1}}{h!}([-\Delta+M_V,S_0] -\im \dot S_0) \\
			&  =-\Delta + [M_V,S_0] + \sum_{h= 2}^\infty \frac{(\ad S_0)^{h-1}}{h!}([M_V,S_0] - M_V),\nonumber
		\end{align}
		so for $\omega\in \cO_0$
		\[
		P_0 : = [M_V,S_0] + \sum_{h= 2}^\infty \frac{(\ad S_0)^{h-1}}{h!}([M_V,S_0] - M_V)= \sum_{h= 1}^\infty \frac{h (\ad S_0)^{h} }{(h+1)!}M_V \,.
		\]
		In  Appendix
		\ref{app:tech.lemma}  we prove that there exist positive constants $C, q_2, q_3$, depending on $\delta, \mu, \tau, \tc, p$ only, such that the following holds.
		\begin{lemma}\label{commuta.bene}
			Let $S_0$  be the solution of the homological equation \eqref{hom.sweet.hom}. Then for {$2a'< \ta$} we have
			\begin{equation}\label{sol.stima}
				\seuno{S_0}{\cO_0}{a'} \leq \frac{C}{{\g (\ta-2a')^{q_2}}} \|V\|_{\ta,\tp}\,.
			\end{equation}	
			Moreover	
			the commutator $[S_0, M_V]$ is quasi-T\"oplitz  of order $-2$ with the bounds
			\begin{equation}\label{comm.hom}
				\sedue{[S_0, M_V]}{\cO_0}{a'} \leq C{\g^{-1}(\ta - 2 a')^{-q_3}} \|V\|^2_{\ta,\tp} 
			\end{equation}
		\end{lemma}
		{	\begin{proof}[Proof of Lemma \ref{passozero}]
				We need to bound $\sedue{(\ad S_0)^{h} M_V}{\cO_0}{a_0}$ for all $h\ge 1$. If $h=1$, this is achieved through \eqref{comm.hom} with $a'=a_0$.
				\\
				If $h=2$, then by \eqref{a.mano} of Proposition \ref{prop.chiavetta} (and recalling that $\seuno{\cdot}{\cO_0}{a}\le \sedue{\cdot}{\cO_0}{a}$)
				\[
				\sedue{[[S_0,[S_0,M_V]]}{\cO_0}{a_0} \le  C\ta ^{-{q_1}}\seuno{S_0}{\cO_0}{\frac\ta3}\sedue{[S_0,M_V]}{\cO_0}{\frac\ta3}\le C^3\ta ^{-{q_1}-q_2-q_3}\g^{-2} \|V\|^3_{\ta,\tp}\,.
				\]
				Otherwise if $h\ge 3$, then we set $A= (\ad S_0)^{h-2}[S_0,M_V]$ so that, by \eqref{stima.chiavetta} and possibly taking $C$ larger,
				\[
				\seuno{A}{\cO_0}{\frac\ta3} \le (C \ta^{-q_0} \seuno{S_0}{\cO_0}{\frac\ta3})^{h-2}\sedue{[S_0,M_V]}{\cO_0}{\frac\ta3}\,.
				\]
				Then we apply \eqref{a.mano},
				which reads
				\[
				\sedue{[S_0,A]}{\cO_0}{a_0}\le  C\ta^{-q_1} (C \ta^{-q_0} \seuno{S_0}{\cO_0}{\frac\ta3})^{h-1}\sedue{[S_0,M_V]}{\cO_0}{\frac\ta3}
				\le  \g (C\g^{-1} \ta^{-\frac{q_4}2} \|V\|_{\ta,\tp})^{h+1}\,
				\]
				with $q_4= 2(q_0+q_1+q_2+q_3)$.
				We deduce that
				\[
				\sedue{P_0}{\cO_0}{a_0}\le \g  \sum_{h\ge 1} \frac{h (C\g^{-1} \ta^{-\frac{q_4}2} \|V\|_{\ta,\tp})^{h+1}}{(h+1)!} \le C'\g^{-1} \ta^{-q_4} \|V\|_{\ta,\tp}^2,
				\]
				provided that  $C\g^{-1} \ta^{-q_4} \|V\|_{\ta,\tp}\le \frac12$. {This amounts to fixing $\epsilon_*\le 2C^{-1}\ta^{q_4}$ in \eqref{pic colo0}.}
		\end{proof}}
		
		\paragraph{The  iterative scheme.}
		{After the first step we have conjugated $\fL$ to an operator of the form $-\Delta +P_0$ where $P_0\in \opdue{\cO}{a_0}$ is appropriately small. Now we apply a slight modification of a standard KAM scheme in order to reduce to constant coefficients. 
			As before the main novelty is that through the iteration we are able to control the quasi-T\"oplitz norm.}
		The KAM reduction scheme is based on the iteration of the following lemma.
		\begin{lemma}[KAM step]\label{stepKam}
			Fix $\g>0$ and $K\in \N$.	Consider a Schr\"odinger operator  $D+P$ with $D$ defined in \eqref{diagonale}  and $P\in \opdue{\cO}a$ satisfying
			\begin{equation}
				\label{piccolostep}
			16 \sedue{\tilde\Omega}{\tR}{a} \le \g , \quad 	16\sedue{P}{\cO}a \le \g K^{-3\tau-1}\ \,.
			\end{equation}
			Let $\cO_+$ be the set defined in \eqref{Opiu}.
			There exists 
			$S\in \opdue{\cO_+}{a}$ and a time independent operator $\tilde\Omega^+ \in \opdue{\tR}{a}$ such that, setting $G=e^S$,  for all $\omega\in \cO_+$ one has
			\[
			G_* (D+P) = D_+ +P_+\,,\quad D_+ = \diag(|j|^2 +\tilde\Omega^+_j)_{j\in \Z^2}
			\]
			where $P_+\in \opdue{\cO_+}{a'}$ for all $0<a'<a$ with the estimate
			\begin{equation}
				\label{stimaP}
				\sedue{P_+}{\cO_+}{a'}\le  C(p,\delta)a^{-2/\delta} \g^{-1} K^{3\tau+1}(\sedue{P}{\cO}{a})^2 + e^{-\frac{a-a'}2K} \sedue{P}{\cO}{a} \ , 
			\end{equation}
			and 
			\begin{equation}\label{stimaautov}
				\sedue{\tilde\Omega-\tilde\Omega^+}{\tR}{a}\le \sedue{P}{\cO}{a} \ . 
			\end{equation}
		\end{lemma}
		\begin{proof}
			We start by defining $S$ as the solution of the homological equation
			\begin{equation}\label{homo equatio}
				-\im \,\dot S +	[D,S] = - \Pi_{0<|\ell|\le K} P\;\iff \; S_j^{j-\pi(\ell)}(\ell)= \begin{cases}
					-\dfrac{P_j^{j-\pi(\ell)}(\ell)}{ \td(\ell,j)} 
					\quad & \mbox{if}\; 0<|\ell|\le K\\
					0 & \mbox{otherwise}
				\end{cases} 
			\end{equation}	
			Note that $S$ is well defined for $\omega\in \cC^{(1)}$ (defined in \eqref{CD1}) and, by Proposition \ref{homolog}, one has $S\in \opdue{\cO_+}{a}$, with the bound \eqref{stimahomol}. Then
			\begin{align}
				\label{Gstar}
				G_* (D+P)&= e^{\ad S } (D+P)-\im \sum_{h= 1}^\infty \frac{(\ad S)^{h-1}}{h!}\dot S \\
				&= (D+P) + \sum_{h= 1}^\infty \frac{(\ad S)^{h-1}}{h!}([D+P,S] -\im \dot S) \nonumber\\
				&  =D + \diag(P_j^j(0)) + \Pi_{>K} P + [P,S] +\sum_{h= 2}^\infty \frac{(\ad S)^{h-1}}{h!}([P,S] -  \Pi_{0<|\ell|\le K} P)\nonumber \ ;
			\end{align}
			so for $\omega\in \cC^{(1)}$ we put 
			\[
			P_+ : =  \Pi_{>K} P + [P,S] +\sum_{h= 2}^\infty \frac{(\ad S)^{h-1}}{h!}([P,S] -  \Pi_{0<|\ell|\le K} P)\,,\quad D_+:=D + \diag(P_j^j(0))_{j\in \Z^2}\,.
			\]
			By construction $P= P^\tT +P^{(1)}+P^{(2)}$ as in \eqref{arco}. Now by Lemmas \ref{proiettotutto} and  \ref{diago} 
			\[
			P_j^j(0) = \mathfrak P(0,v(j),b(j)) +( P^{(1)})^j_j(0) + ( P^{(2)})^j_j(0)\,.
			\]
			So in order to define $D_+$ on all of $\tR$ we just extend each $ \mathfrak P(0,v,b),( P^{(i)})^j_j(0)$ by Kirszbraun theorem preserving its weighted Lipschitz norm $|\cdot |^\gamma$. The bound \eqref{stimaautov} follows.
			The bound \eqref{stimaP} follows  by Corollary \ref{esponenziale} and Lemma \ref{proiettotutto}. 
		\end{proof}
		\begin{proposition}[{Iteration}]\label{ittero}
			Let $\fL_0$ be a Schr\"odinger operator of the form
			\begin{equation}
				\label{scrozero}
				-\Delta +P_0\,,\quad -\Delta=\diag (|j|^2) \,,\quad P_0\in \opdue{\cO_0}{a_0}\,,\quad \e_0:=\g^{-1} \sedue{P_0}{\cO_0}{a_0}<K_0^{-3\tau-2}
			\end{equation}
			where $K_0$ is large enough.
			Set
			\begin{equation}
				\label{costantine}
				a_n= a_0(1-\sum_{k=0}^{n-1} 2^{-k-2})\,,\quad K_n:= K_0 4^n\,,\quad \e_n = \e_0 e^{-(\frac32)^n+1}.
			\end{equation}
			For all $n\ge 1$, there exist
			\begin{enumerate}
				\item[$(i)_n$] A sequence of time-independent operators 
				\begin{equation}\label{stima Om}
					D_n = -\Delta+ \tilde\Omega^{(n)}\,,\quad \mbox{where}\quad
					\tilde\Omega^{(n)}\in \opdue{\tR}{a_n}\quad \mbox{with}\quad \sedue{\tilde\Omega^{(n)}-\tilde\Omega^{(n-1)}}{\tR}{a_n} \le \g \e_{n-1}
				\end{equation}
				(we have set $\Omega^{(0)}:=0$).
				\\
				Following Lemma \ref{diago} we denote the eigenvalues as
				\begin{equation}
					\label{nominomi}
					\tilde\Omega_j^{(n)}= \fa^{(n)}(v(j),b(j))+ r^{(n,1)}_j + r^{(n,2)}_j \, 
				\end{equation} and define the  nested sequence of compact sets (recall \eqref{CD1}, \eqref{CD2})
				\begin{equation}
					\label{On} 
					\cO_n:=\cO_{n-1}\cap \cC^{(1)}_{D_{n-1},K_{n-1}}(\g) \cap \cC^{(2)}_{D_{n-1},K_{n-1}}(\g) \ . 
				\end{equation}
				\item[$(ii)_n$] A sequence of  operators $P_n\in \opdue{\cO_n}{a_n}$, $S_n\in \opdue{\cO_n}{a_{n-1}}$ such that 
				\begin{equation}\label{transfo}
					\left(e^{S_{n}}\right)_*(D_{n-1}+P_{n-1})= D_n +P_n \,,\quad \forall \omega \in \cO_n,
				\end{equation}
				\begin{equation}\label{pic colon}
					\g^{-1}\sedue{P_n}{\cO_n}{a_n}\le \e_n\,,\quad  \sedue{S_n}{\cO_n}{a_{n-1}}\le K_{n-1}^{3\tau+2}\e_{n-1}.
				\end{equation}
			\end{enumerate}
		\end{proposition}
		\begin{proof}
			The lemma is proved by standard iterative estimates.
		\end{proof}
		\begin{corollary}\label{chiudo}
			Given a Schr\"odinger operator $\fL_0$ satisfying \eqref{scrozero} there exist a time independent operator $\tilde\Omega\in \opdue{\tR}{a_0/2}$,   a Cantor-like set $\cO_\infty$ 
			and a  symplectic change of variables  $G$ with $G-\id \in \opdue{\cO_{\infty}}{a_0/2}$ such that
			\begin{equation}
				\label{trova}
				G_* \fL_0=  -\Delta+\tilde{\Omega}\,, \quad \forall \omega\in \cO_\infty.
			\end{equation}
		\end{corollary}
		\begin{proof}
			The proof of this lemma is also standard and follows from the fast convergence of the $\tilde\Omega^{(n)}$ in Proposition \ref{ittero}. We put $\cO_\infty:= \cap_n \cO_n$.
		\end{proof}
		\begin{proof}[Proof of Theorem \ref{riduco}]
			The thesis follows from Lemma \ref{passozero} and Corollary \ref{chiudo}, except from the fact that $\cO_\infty$ has positive measure. In order to show this, we shall prove
			Proposition \ref{prop:Ofin}, which gives a cleaner characterization of the set  $\cO_\infty$. Then, the measure estimates are deferred to the next section.
		\end{proof}
	{\begin{proof}[Proof of Proposition \ref{prop:Ofin}]
				We need to show that $\cO_{\rm fin}\subset \cap_{n=0}^\infty \cO_n$, i.e.\ we have to verify conditions \eqref{CD1}-\eqref{CD2} for $\omega\in\cO_{\rm fin}$.  Proving  \eqref{CD1} is standard (see for instance \cite{BB}). Regarding 
 \eqref{CD2}, by \eqref{CD2fin} one has, for $\omega\in\cO_{\rm fin}$,
			\begin{align*}
		&	\big|\omega\cdot\ell + 2\frac{|\pi(\ell)|}{|v|}b - |\pi(\ell)|^2 + \fa^{(n)}(v,b) - \fa^{(n)}(v,b-v\cdot\pi(\ell))\big| 
			 \geq \\
			& 4\gamma |\ell|^{-\tau}- |\fa^{(n)}(v,b) - \fa(v,b)| - |\fa^{(n)}(v,b-v\cdot\pi(\ell))-\fa(v,b-v\cdot\pi(\ell))|.
			\end{align*}
		The estimate follows by  recalling that $|\ell|<K_n$, using \eqref{stima diago} to bound $|\fa^{(n)}(v,b) - \fa(v,b)|$ in terms of $\sedue{\tilde\Omega^{(n)}-\tilde\Omega}{\tR}{a_n}$ and then exploiting the bounds \eqref{stima Om}
		 (see e.g. the proof of Lemma 7.6 of \cite{MP} for details).
		\end{proof}}}
		\section{Measure estimates}
%
%
%
%
The purpose of this section is to prove that the set of frequencies $\omega$ for which the non-resonance conditions \eqref{N.melnikov} hold up to order $N$  has positive measure.
\\
	Recalling the notations of Theorem \ref{riduco} and denoting
\[
\Omega_j= |j|^2 +\fa(v(j),b(j)) + R(j)\,,\quad R(j)= r^{(1)}_j + r^{(2)}_j \,,
\]
we have
\begin{equation}
\label{mel.decay}
\sup_{v,b} e^{\ta' \tc |v|}\jap{b}^2 |\fa(v,b)|^\tR +\sup_j |R(j)|^\tR \jap{j}^\mu \le 2 \epsilon \gamma\,.
\end{equation}
	Denote by 
	\[\bv := (v_h)_{h =1, \ldots, N} \in \cV^N\,\quad \bb := (b_h)_{h =1, \ldots, N} \in \Z^N\,,\quad \pmb{\eta}:=(\eta_{ab})_{\substack {a=1,2\\b=1,\ldots, N}} \in \{-1,0,1\}^{2N}
	\]	
	and for $(\ell,K,\bv,\bb,\pmb{\eta})\in \fA_N:= \Z^d\times \Z\times \cV^N\times \Z^N\times \{-1,0,1\}^{2N}$, consider an expression of the form
	\begin{equation}\label{3meln}
		\begin{aligned}
			\td(\ell,K,\bv,\bb,\pmb{\eta}):=	\omega & \cdot \ell + K   +\sum_{h = 1}^N \eta_{1h}\, \fa(v_h,b_h)   +\sum_{k=1}^N \eta_{2k}\,  R(j_k)
		\end{aligned}
	\end{equation}
	We shall prove the following result:
	\begin{lemma}\label{lem.N.melnikov}
	Given $N \in \N$, $N \geq 2$, 
		there exist $\gamma_*, \epsilon_*,\tau_N >0$ such that for any $0< \gamma < \gamma_*$ and $\epsilon<\epsilon_*$ the set
		\begin{equation}
			\label{cC_*}
			\cC_*(N,\gamma, \tau_N) := 	\cC_* := \left\{ \omega \in \tR \colon  \ 
			\abs{\td(\ell,K,\bv,\bb,\pmb{\eta})} \geq \frac{\gamma}{\la \ell \ra^{\tau_N}} , \ \ \forall (\ell,K,\bv,\bb,\pmb{\eta})\in \fA_N :\; (\ell,K) \neq (0,0) \right\}
		\end{equation}
		has positive measure. More precisely there exists a constant $c_* >0$ such that
		$$
		\meas (\tR \setminus \cC_*)\leq c_* \gamma  \ .
		$$
	\end{lemma}
	\begin{proof}
		We first notice that for all $\tL >0$, we can ensure that
		$$
		\abs{\eqref{3meln}} \geq \frac{\gamma}{\la \ell \ra^{d+1}} , \quad \forall (\ell, K) \neq (0,0), \ \ |\ell| \leq \tL   
		$$
		just by taking $\epsilon_* = \epsilon_*(\tL)$ small enough.
		So we choose $\tL >0$ (to be fixed later).
		Next we note that, for $\ell\ne 0$,  if $|K|> 2|\ell|$ then (for $\epsilon_*$ small enough) one has
		$|\td(\ell,K,\bv,\bb,\pmb{\eta})|\ge \frac12>\g$.

		So from now on we restrict to $|\ell| \geq \tL$ and   $|K| \leq 2|\ell|$.\\
		We prove the claim by finite induction on the number of $\eta_{r_1, r_2}$ different from 0. 
		More precisely for every $0\leq n \leq 2N$ we shall show that for $\gamma$ small enough, there exist a  sequence of increasing $\tau_n$ and a sequence of nested sets $\cC^n = \cC^n(\gamma, \tau_n)$ such that, provided
		$$
		|\eta_{11}| + \cdots + |\eta_{23}| = n \ , 
		$$
		then 
		for any $\omega \in \cC^n$ one has
		\begin{equation}\label{meln.ind}
			\abs{\eqref{3meln}} \geq \frac{\gamma}{\la \ell \ra^{\tau_n}} , \ \ \forall (\ell,K) \neq (0,0) 
		\end{equation}
		and moreover 
		\begin{equation}
			\label{meln.mes}
			\meas (\tR \setminus \cC^0)\leq C \gamma  \ , \qquad
			\meas (\cC^n \setminus \cC^{n+1})\leq C \gamma 
		\end{equation}
		for some $C>0$. Then the lemma follows by taking $\cC_* := \cC^{2N}$ and $c_*:= (2N+1) C$.  Then $\gamma_*$ is fixed so that $(2N+1)C \gamma_* < \meas (\tR) $, and the measure of $\cC_*$ results positive.\\
		
		{\em Case $n=0$.} In this case \eqref{3meln} reduces to $\omega \cdot \ell + K$, so the set $\cC^0$ coincides with the set $\cO_0$ defined in \eqref{diofantino}, with $\tau_0:= d+1$. It is well known that $\meas{(\tR \setminus \cO_0)} \leq C \gamma$ for some $C >0$.\\
		
		{\em Case $n \leadsto n+1$.} Assume that \eqref{meln.ind} holds for any possible choice of $(\eta_{ab})$ s.t. $|\eta_{11}|+\cdots + |\eta_{23}| = n$, $1\leq n \leq 2N-1$, for some $(\tau_a)_{a=1,\ldots, n}$.  We can assume that $|\ell|>\tL$ and $|\eta_{11}|+\cdots + |\eta_{23}| = n+1 $.
		\\
		We further divide two subcases. 
		
		\underline{Case I: } At least one of the following conditions holds:
		
		\begin{itemize}
			\item[I.a] There exists $h=1,\dots, N$ such that $\eta_{1h}\ne 0 $ and either $ {\tc \ta' }|v_h| \geq |\ell|$ or $|b_h|\ge |\ell|^{\tau_n/2}.$
			
			\item[I.b]  There exists $h=1,\dots,N$ such that $\eta_{2h}\ne 0 $ and $\jap{j_h} \ge |\ell|^{\tau_n / \mu}.$
		\end{itemize}
		In Case I.a, assume w.l.o.g. that $\eta_{1N} \neq 0$ then 
		\begin{align*}
			| \eqref{3meln}|& \geq \abs{\omega \cdot \ell + K + \sum_{a=1}^{N-1} \eta_{1a} \fa(v_a,b_a) + \sum_{a=1}^N \eta_{2a} R(j_a) } - |\fa(v_N, b_N)| \\
			& \geq \frac{\gamma}{\la \ell \ra^{\tau_n}} - 2 \g \epsilon \jap{b_N}^{-2}e^{-\tc \ta' |v_N| } \ge \frac{\gamma}{\la \ell \ra^{\tau_n+1}}
		\end{align*}
		which is true provided $\epsilon_*$ is small enough and $\tL$ sufficiently large.
		In Case I.b, assume w.l.o.g. that $\eta_{2N} \neq 0$ then 
		\begin{align*}
			| \eqref{3meln}|& \geq \abs{\omega \cdot \ell + K + \sum_{a=1}^N \eta_{1a} \fa(v_a,b_a) + \sum_{a=1}^{N-1} \eta_{2a} R(j_a) } - |R(j_N)| \\
			& \geq \frac{\gamma}{\la \ell \ra^{\tau_n}} - \g \epsilon \frac{1}{\jap{j_N}^\mu} \ge \frac{\gamma}{\la \ell \ra^{\tau_n+1}}
		\end{align*}
		%
			
			\underline{Case II: } For all $k=1,2$ and $h=1,\dots,N$ such that $\eta_{kh}\ne 0$ one has  ${\tc \ta' }|v_h| < |\ell|$ and $|b_h| < |\ell|^{\tau_n/2}$ if $k=1$    and one has $\jap{j_h} < |\ell|^{\tau_n/\mu}$ if $k=2$. 
			Now for  $|\ell|\ge \tL$, $|K|\le 2|\ell|$ and $\bv,\bb,\pmb{\eta}$ as in Case II we define
			\[
			\cR_{\ell,K,\bv,\bb,\pmb{\eta}}(\tau_{n+1}, \gamma):=\{\omega\in \tR:\quad |\td(\ell,K,\bv,\bb,\pmb{\eta})| \le \gamma |\ell|^{-\tau_{n+1}}\} \ . 
			\]
			Now  (recall $|\ell|\ge \tL$ and the Lischitz estimates on $\fa(v,b)$ and $R(j)$)
			\[
			|\frac{d}{ d|\omega|} \td(\ell,K,\bv,\bb,\pmb{\eta})| \ge |\ell|-2N \epsilon  \ge 1
			\]
			so 
			$
			\mbox{meas}\, (  \cR_{\ell,K,\bv,\bb,\pmb{\eta}}) \le 2\gamma |\ell|^{-{\tau_{n+1}}}
			$
			and
			\[
			\meas \left(\bigcup_{\substack{|\ell| \geq \tL, \ |K|\le 2|\ell|
					\\
					\bv,\bb,\pmb{\eta} \in \mbox{Case II}}}
			\!\!\!\!\!  \cR_{\ell,K,\bv,\bb,\pmb{\eta}}	
			\right) \le
			\sum_{|\ell| \geq \tL} \ \sum_{|K| \leq 2 |\ell| , \ |v_h|\leq |\ell| \atop 
				|b_h| \leq |\ell|^{\tau_n/2}, \ \la j_h \ra \leq |\ell|^{\tau_n/\mu}}  	
			\frac{\gamma }{|\ell|^{\tau_{n+1}}} \leq C
			\gamma \sum_{|\ell| \geq \tL}  |\ell|^{2N+1+(\frac12 + \frac{2}{\mu})N \tau_n - \tau_{n+1} }
			\]
			which is finite provided
			$$
			\tau_{n+1} > (\frac52 + \frac{2}{\mu})N \tau_n+ d +1.
			$$
			Then we put
			$$
			\cC^{n+1}:= \cC^n \setminus \bigcup_{\substack{|\ell| \geq \tL, \ |K|\le 2|\ell|
					\\
					\bv,\bb,\pmb{\eta} \in \mbox{Case II}}}  \cR_{\ell,K,\bv,\bb,\pmb{\eta}}
			$$
			and clearly it fulfills \eqref{meln.mes}. This concludes the proof of the inductive step.
			
		\end{proof}
		
		We are now ready to discuss the measure of the non-resonant $\omega$. Let us start by discussing the set $\cO_{\rm fin}$ of Proposition \ref{prop:Ofin}.
			\begin{proposition}\label{misura1}
There exist $\gamma_*, \epsilon_*,\tau >0$ such that for any $0< \gamma < \gamma_*$ and $\epsilon<\epsilon_*$, the set $\cO_{\rm fin}$ of Proposition \ref{prop:Ofin} has positive measure, in particular
			$$
			\textup{meas}(\tR \setminus \cO_{\rm fin}) \leq C \gamma \ .
			$$
		\end{proposition}
		\begin{proof}
			We show that the set $\cC_*$ defined in \eqref{cC_*} with $N=2$ is contained in $\cO_{\rm fin}$.
			Indeed  $\cC^{(1)}$ defined in \eqref{CD1fin} contains  $\cC_*$, just taking
			$j_1=j$, $j_2=j-\pi(\ell)$, $(v_i,b_i)= (v(j_i),b(j_i))$ for $i=1,2$,
			 $K=|j_1|^2-|j_2|^2$, $\eta_{11} = \eta_{21} = 1$, $\eta_{12}=\eta_{22} = -1$.
			\\
			Also $\cC^{(2)}$ defined in \eqref{CD2fin} contains  $\cC_*$, setting $v_1=v_2= v$, $b_1=b$, $b_2= b- v\cdot\pi(\ell)$, 
			$K= 2 \frac{|\pi(\ell)|}{|v|} b -|\pi(\ell)|^2$ (note that $\frac{|\pi(\ell)|}{|v|}b \in \Z$  as $\pi(\ell) \parallel v $) and taking
		$\eta_{11}=-\eta_{12}$,	$\eta_{21} = \eta_{22}= 0 $.
			
		\end{proof} 
%
%

We finally prove Theorem \ref{riduco.circa}:
\begin{proof}[Proof of Theorem \ref{riduco.circa}]
Reducibility is a consequence of Theorem \ref{riduco}. It only remains to prove that $\forall N \in \N$ Melnikov conditions of order $N$ are satisfied, namely that 
 $\cC_* \subseteq \cG_N$, with $\cG_N$ as in \eqref{G.N} and $\cC_*$ as in \eqref{cC_*}. First we observe that, since $L$ as in \eqref{G.N} satisfies $|L| \leq N$, it has at most $N$ non vanishing components $L_{j_1}, \dots, L_{j_N}$. Then it is sufficient to set
\begin{gather}
\nonumber
\bv = (\underbrace{v(j_1), \dots, v(j_1)}_{|L_{j_1}| \textrm{ times}}, \dots, \underbrace{v(j_N), \dots, v(j_N)}_{|L_{j_N}| \textrm{ times}})\,, \quad \bb = (\underbrace{b(j_1), \dots, b(j_1)}_{|L_{j_1}| \textrm{ times}}, \dots, \underbrace{b(j_N), \dots, b(j_N)}_{|L_{j_N}| \textrm{ times}})\,,\\
\label{def.K}
\eta = (\underbrace{\textrm{sign}(L_{j_1}), \dots, \textrm{sign}(L_{j_1})}_{|L_{j_1}| \textrm{ times}}, \dots, \underbrace{\textrm{sign}(L_{j_N}), \dots, \textrm{sign}(L_{j_N})}_{|L_{j_N}| \textrm{ times}})\,, \quad K = \sum_{j \in \Z^2} |j|^2 L_{j}\,.
\end{gather}
Then the result follows by Lemma \ref{lem.N.melnikov}, observing that the set $\cG_N$ is defined in such a way that, with $K$ as in \eqref{def.K}, one always has $(\ell, K) \neq (0,0)$.
\end{proof}
	
	\section{Nonlinear stability results}\label{3birk}
	In order to prove stability we apply a classical Birkhoff normal form procedure, see for instance Theorem 8.1 of \cite{MP}.
	We start by recalling some standard notions on 
	 time dependent Hamiltonians. 
	 \\
	 One works in the {\it extended phase space} $(\f,\cY,u)\in \T^d\times \R^d\times \fH^p$ with the symplectic form 
	$d\cY\wedge d\f +\im d u\wedge d\bar u$.
\\
In this context we shall consider {\em real on real} analytic  functions on  the complex domain
\begin{align*}
D(a,r):=\T^\tk_a\times D(r)\,,\quad 
& D(r)  :=\left\lbrace \yy\in \C^\tk:\; 
|\yy|_1:=\sum_{i=1}^\tk|\yy_i| < r^2 \,,
\quad
u\in \fH^p : \;
\|u\|_p < r \right\rbrace  \,,
\end{align*}
with  $\T^d_a$  defined in \eqref{toro}.
 Given a real  analytic function $F(\f,\cY,u)$ we consider its Taylor-Fourier series
	\[
F(\f, \cY,u)= \sum_{\substack{\al ,\bt\in\N^{\Z^2}\\\ell\in \Z^\tk, l\in \N^\tk}} F_{\al,\bt,l,\ell} \, 
 e^{\im \ell\cdot \f} \, \yy^l \,  u^\alpha \, \bar u^\beta \,,\quad F_{\al,\bt,l,\ell}= \bar F_{\bt,\al,l,-\ell}\,,
\quad	u^\alpha \bar u^\bt= \Pi_{j \in \Z^2} u_j^{\alpha_j}\bar u_j^{\bt_j}
\]
		Correspondingly we expand vector fields in Taylor Fourier series (again  well defined  and pointwise absolutely convergent):
	$$
	X^{(w)}(\f,\yy, u)= \sum_{\al ,\bt\in\N^{\Z^2},\ell\in \Z^\tk, l\in \N^\tk } X_{\al,\bt,l,\ell}^{(w)} \, e^{\im \ell\cdot \f} \, \yy^l \,  u^\alpha \, \bar u^\beta\,,
	$$
	where $w$ denotes the components $\f_i, \yy_i$ or $u_\jj,\bar u_\jj $.
	To a vector field we associate its {\em majorant }
	\begin{equation}\label{seriamente}
	\und{X}_a^{(w)}[\yy,u]:= \sum_{\ell\in \Z^\tk,l\in \N^\tk,\al ,\bt\in\N^{\Z^2} } |X^{(w)}_{\al,\bt, l, \ell}| \, e^{a\, |\ell|} \, \yy^l \,  u^\alpha \, \bar u^\beta  \ .
	\end{equation}
	Then we have the following
	\begin{definition}
		A vector field $X: D(a,r) \to  \C^{\tk}\times \C^\tk\times \fH^p$ will be said to be {\em majorant analytic} in $ D(a,r)$  if $\und{X}_a$ defines an analytic vector field  $D(r)\to \C^\tk\times\C^\tk\times \fH^p$.
	\end{definition}
	Since Hamiltonian functions are defined modulo constants, 
	we give the following definition of {\em regular} Hamiltonian and its   norm:
	\begin{definition}
		A real valued Hamiltonian $H$ will be said to be {\em regular} in $D(a,r)$ if its Hamiltonian vector field $X_H$ is majorant analytic in $D(a,r)$. We define its {\em norm} by
		\begin{equation}\label{musically}
		|H|_{a,r}:=\sup_{(\yy,u)\in D(r)} \bnorm{ \und{(X_H)}_a(\yy, u)}_{r} \,,\quad 
		\bnorm{ X}_r:=|X^{(\f)}|_{\infty}+\frac{|X^{(\yy)}|_1}{r^2}+ \frac{\|X^{(u)}\|}{r}+ \frac{\|X^{(\bar u)}\|}{r} \ .
		\end{equation}
	\end{definition}
	Note that the norm $| F |_{a,r}$ controls the norm $\bnorm{X_F}_r$  in $D(a,r)$.
	
	\smallskip
	
 If $F(\omega; \f, \yy, u) \equiv F(\omega)$ depends on the frequency $\omega \in \cO$, {where $\cO$ is some compact set},  we define the  {\em weighted Lipschitz} norm:
	\begin{equation}
	\label{lip.norm0}
	|F|_{a,r}^\cO:= \sup_{\omega \in \cO}|F(\omega, \cdot)|_{a,r}+  \g\sup_{\omega_1\neq \omega_2\in \cO} \frac{|F(\omega_1, \cdot)-F(\omega_2, \cdot) |_{a,r}}{|\omega_1-\omega_2|} \ . 
	\end{equation}
\begin{definition}
	We say that a regular Hamiltonian $H$ is Gauge and Momentum preserving if
	 \[  \pi(\ell)+ \sum_{j\in \Z^2} j (\al_j-\bt_j)\ne 0 \quad  \mbox{or} \quad   \sum_{i=1}^d \ell_i+ \sum_{j\in \Z^2}  (\al_j-\bt_j)\ne 0  \quad \Rightarrow \quad   H_{\al,\bt, l, \ell} =0
	 \]
	 By convention we define the {\em scaling degree} of a monomial $ e^{\im \ell\cdot \f} \, \yy^l \,  u^\alpha \, \bar u^\beta$ as $
	 \deg(l, \alpha, \beta) := 2| l | + |\alpha| + |\beta| -2\,,$
	 and define the {\em projection} on the homogeneous components of {\em scaling degree} $d$ as
	 $$
	 H^{(d)} := \sum_{\ell\in \Z^\tk, l \in \N^\tk , \al ,\bt\in\N^{\Z^2}  \atop 
	 	2| l |+|\al|+|\bt|=d+2}  \, H_{\al,\bt, l, \ell} \, e^{\im \ell\cdot \f} \yy^l u^\alpha \, \bar u^\beta \ , 
	 $$
	 similarly for $H^{(\leq d)}$ and $H^{(\geq d)}$. One easily verifies that $\{H^{(d_1)}, H^{(d_2)} \}$ has scaling degree $d_1+d_2$.
\end{definition}
	\paragraph{Proof of Theorem \ref{cubic}.}
	A direct computation ensures that the NLS equation \eqref{NLS3} is Hamiltonian 
		\begin{equation}
		\label{NLSf}
		H= \omega\cdot \cY + \sum_j |j|^2 |u_j|^2 + (u, V(\f) u)_{\ell_2(\Z^2)} + P(\f,u) \,,\quad P= \int_{\T^2} P(\f +\bK x, u(x),\bar u(x)) dx\,.
		\end{equation}
	By assumptions 
	$P$ is a Gauge and Momentum preserving regular Hamiltonian of scaling $\ge 1$, with norm 
	\[
	|P|_{a,r} \le C r \,,\quad \forall r\le r_* \ . 
	\]
	We now fix $0<\g < \g_*$ and  assume that $\|V\|_{\ta,\tp}\g^{-1}<\epsilon_*$ so that we may apply Theorem \ref{riduco}. The time dependent symplectic change of variables $G$ defined in Theorem \ref{riduco} gives rise to the simplectic  trasformation $\cG$ defined in \eqref{simplettica} on the extended phase space. By definition,  $\cG$  preserves the scaling degree and maps $D(\ta/8,r) \to D(\ta/8,\rho r)$ for all $r>0$, where $\rho-1 \lesssim \|V\|_{\ta,\tp}\g^{-1}$.
		 Thus   for all regular $F$ of scaling $\ge 1$   one has that $F\circ \cG$ is still regular, of scaling $\ge 1$,  and satisfies
		\[
		|F\circ \cG|_{a,\rho r} \le  2|F|_{a,r} \,,\quad \forall a\le \ta/8.
		\]
		Finally, since $V$ is Gauge invariant  then $\cG$ is Gauge covariant.
		 By construction   $\cG$  reduces the Hamiltonian $H$ to the diagonal form
	\[
H_1=	H\circ \cG = \omega\cdot \cY + \sum_j \Omega_j |u_j|^2 + \cK(\f,u) \,,\quad  |\cK|_{\ta/8,r} \le C r \,,\quad \forall r\le r_*\,,
	\]
	where $\cK$ is a Gauge and Momentum preserving regular Hamiltonian of scaling $\ge 1$.
	We note that, thanks to Theorem \ref{riduco.circa} with $N=3$, we can impose third order Melnikov conditions of the forma \eqref{N.melnikov} for any $(\ell, L) \in \cG_3$.
	Since $\cK$ is a regular  Gauge and momentum preserving Hamiltonian, we can eliminate all of its monomials of scaling degree 1. 
	Precisely, as $H_1$ fits the hypotheses of Theorem 8.1 of  \cite{MP} (with $\e =1$), 
 there exists a symplectic change of variables which cancels the terms of scaling degree one in $\cK$ and conjugates $H_1$ to
\[
H_2= \omega\cdot \cY + \sum_j \Omega_j |u_j|^2 + \widetilde\cK(\f,u) \,,\quad  | \widetilde\cK|_{\ta/8,r} \le C r^2 \,,\quad \forall r\le r_* \ , 
\]
where $\widetilde\cK$ has scaling $\ge 2$. Then the stability times in Theorem \ref{cubic} follow by a standard contraction mapping Lemma argument, see for instance Corollary 5.1  of \cite{BMP:2019}.
	\appendix
{	
	\section{Technical lemmas}\label{sec:technical}
		\begin{lemma}[Bony] \label{Bony}
		Given $M \in \cL_a(\fH^p)$ and setting
		\begin{equation}
			\label{eq:bony}
			(M^{(B)})_j^{j'}(\ell):= \begin{cases}
				0 \quad {\rm if} \quad  |\ell| > \tc \jap{j}^\delta\\
				M_{j}^{j'}(\ell) \quad {\rm otherwise} 
			\end{cases}\,,\quad M^{(R)} = M- M^{(B)}
		\end{equation}
		we have that $ M^{(R)} \in \cL_{a', -n}$ for  any $n>0$ and any $0 \leq a'<a$ with the quantitative estimate 
		\begin{equation}\label{MR.est}
			| M^{(R)}|_{a',-n} \lesssim_{p,\delta}  \frac{| M|_{a}}{(a-a')^{\mu n/\delta}}
		\end{equation}
	\end{lemma}
	\begin{proof}
		Exploiting the definition \eqref{due.pesi} we have
		\begin{align*}
			(	\underline M^{(R)}_{a',-n}) _{j}^{j'} &:= \sum_{ \ell:  j-j'= \pi(\ell)\atop  |\ell|>\tc \jap{j}^\delta } e^{a'|\ell|} |M_{j}^{j'}(\ell)|\jap{j}^{\mu n} \le
			\tc^{-\mu n/\delta}\sum_{ \ell : j-j'= \pi(\ell)} e^{a|\ell|} e^{-(a-a')|\ell|}|\ell|^{n\mu/\delta}  |M_{j}^{j'}(\ell) |\\
			&\le   C(\mu n/\delta,a-a') (\underline M_a)_{j}^{j'}\,,
		\end{align*}
	with $C(p, \sigma)$ defined as in \eqref{def.C}
	 for $p, \sigma>0$.
	\end{proof}
	\begin{proof}[Proof of Lemma \ref{lemma.algebra}]
		To prove \eqref{stima1}, we just use the fact that one has
		\begin{align*}
			(\underline{M N}_{a,-\bN_1}) _{j}^{j'} &:= \sum_{ \ell_1, \ell_2:  j-j'= \pi(\ell_1+\ell_2) } e^{a|\ell_1+\ell_2|}\,  |M_{j}^{j-\pi(\ell_1)}(\ell_1)|\, \jap{j}^{\mu n_1} \,  |b(j)|^{m_1} \, |N_{j-\pi(\ell_1)}^{j'}(\ell_2) |\\
			&\le \sum_{j_1} (\underline{M}_{a,-\bN_1}) _{j}^{j_1} (\underline{N}_{a,\vec 0}) _{j_1}^{j'}\,.
		\end{align*}
		For the second statement we divide $M = M^{(B)} +M^{(R)}$ as in \eqref{eq:bony}, then
		\begin{align*}
			& \quad 	(\underline{M^{(B)} N}_{a;-(n_1 + n_2, 0)}) _{j}^{j'} {:= \sum_{ \ell_1, \ell_2:  j-j'= \pi(\ell_1+\ell_2), \atop |\ell_1|< \tc\jap{j}^\delta } e^{a|\ell_1+\ell_2|} \,
			|M_{j}^{j-\pi(\ell_1)}(\ell_1)| \, 
			\jap{j}^{\mu(n_1+ n_2)}\, 
			|N_{j-\pi(\ell_1)}^{j'}(\ell_2)|}\\ 
			& \le 
			\sum_{ \ell_1, \ell_2:  j-j'= \pi(\ell_1+\ell_2),\atop |\ell_1|<\tc \jap{j}^\delta} \Big( e^{a|\ell_1|} |M_{j}^{j-\pi(\ell_1)}(\ell_1)|\jap{j}^{\mu n_1} \,  e^{a|\ell_2|} \jap{j-\pi(\ell_1)}^{\mu n_2} |N_{j-\pi(\ell_1)}^{j'}(\ell_2) | \Big) \left(\frac{\jap{j}}{\jap{j-\pi(\ell_1)}}\right)^{\mu n_2}\\
			& \le 2^{\mu n_2}\sum_{j_1} (\underline{M^{(B)}}_{a;-\bN_1}) _{j}^{j_1} (\underline{N}_{a; -\bN_2})_{j_1}^{j'}\,.
		\end{align*}
		Moreover by \eqref{stima1} with $M\rightsquigarrow M^{(R)}$ and estimate \eqref{MR.est} we have
		\[
		| M^{(R)} N|_{a'; -(n_1 + n_2, 0)} \le  |M^{(R)}|_{a'; -(n_1 + n_2, 0)} |N|_{a'; \vec 0 } \lesssim  \frac{|M|_{a }|N|_{a'}}{(a-a')^\frac{\mu n}{\delta}}\,.
		\]
	\end{proof}}
\section{Proof of Proposition \ref{prop.chiavetta} and Lemma \ref{commuta.bene}}\label{sec:chiavetta}
\label{app:tech.lemma}
Before proving Proposition \ref{prop.chiavetta}, we give the following auxiliary result:
\begin{lemma}\label{tigre.contro.tigre}
	Given $n_1, m_1, m_2 \in \N$, let $\bN_1 = (n_1, m_1)$, $R\in \cL_{\frac{a}{2}; -\bN_1}$ and $T \in \cT_{a,-m_2}$.
	Then $R T \in \cL_{\frac{a}{2}; -\bN_1}$, and 	 $T R\in \cL_{\frac{a}{2}; -(n_1,m_2)}$, with the bounds
	\begin{eqnarray}\label{easy.one}
	&	|R T|_{\frac{a}{2}; -\bN_1} \lesssim a^{-p-2}  |R|_{\frac{a}{2}; -\bN_1} |T|^\tT_{a,0}\,,
\\
\label{easy.two}
	&	|T R|_{\frac{a}{2}; -(n_1,  m_2)} \lesssim_{\mu, n_1} \, a^{-\mu n_1}|R|_{\frac{a}{2}; -(n_1, 0)} |T|^{\mathtt{ T}}_{a, -m_2}\,.
	\end{eqnarray}
\end{lemma}
\begin{proof}
	Bound  \eqref{easy.one} follows directly from Lemma \ref{lemma.algebra} and Lemma \ref{vaccab}.
	Analogously, to prove \eqref{easy.two} one observes that
	\begin{align*}
		\hspace{-3pt}(\underline{T R}_{\frac{a}{2}; -(n_1, m_2)}) _{j}^{j'} &:= \!\!\!\!\!\!\!\!\sum_{ \ell_1, \ell_2: \atop  j-j'= \pi(\ell_1+\ell_2) } \!\!\!\!\!\!\!\! e^{\frac{a}{2}|\ell_1+\ell_2|} |T_{j}^{j-\pi(\ell_1)}(\ell_1)|\, |  R_{j-\pi(\ell_1)}^{j'}(\ell_2) | \, \jap{j}^{\mu n_1} |b(j)|^{m_2} \\
		&\lesssim_{n_1, \mu} a^{-\mu n_1} \!\!\!\!\!\!\!\!\sum_{ \ell_1, \ell_2: \atop  j-j'= \pi(\ell_1+\ell_2) } \!\!\!\!\!\!\!\!e^{\frac 34 a|\ell_1|} |T_{j}^{j-\pi(\ell_1)}(\ell_1)| |b(j)|^{m_2} e^{\frac{a}{2}|\ell_2|} \jap{j-\pi(\ell_1)}^{\mu n_1}  |R_{j-\pi(\ell_1)}^{j'}(\ell_2) |\\
		&\lesssim_{\mu, n_1} a^{-\mu n_1} 
		\sum_{j_1} (\underline{T}_{\frac 34 a, 0, -m_2})_{j}^{j_1} (\underline{R}_{\frac{a}{2}, -n_1, 0})_{j_1}^{j'}\,,
	\end{align*}
	and the result follows  by Lemma \ref{lemma.algebra}, by Remark \ref{rem:maj.norm}, and by Lemma \ref{vaccab}.
\end{proof}

\begin{proof}[Proof of Proposition \ref{prop.chiavetta}]
	We start by proving estimate \eqref{stima.chiavetta}.\\
	As  $M_\iota \in \cA_{a,-m_i}$, $\iota = 1,2$,    there exist operators $M^\tT_{\iota}, M^{(1)}_{\iota}, \dots, M^{(m_\iota)}_{\iota}$ such that 
	$M_\iota = M^\tT_{\iota} + \sum_{i  = 1}^{m_\iota} M_\iota^{(i)}$ and
	\begin{equation}
		\label{Miota.est}
		|M^\tT_\iota|^{\mathtt {T}}_{a,-m_i} + \sum_{i=1}^{m_\iota} \nor{M^{(i)}_{\iota}}{\frac{a}{2}}{-(i, m_\iota- i )}  \le  {2} \sega{M_\iota} {a}{-m_\iota}\,.
	\end{equation}
	To each $M^\tT_{\iota}$ we may associate the symbol $\fT_\iota(\ell,v,b)$ .
	First of all we define the line-T\"oplitz part of $M_1 M_2$  by setting $(M_1M_2)^\tT:= T$ with 
	\[
	(T)_j^{j'}(\ell):=  \fT(\ell,v(j),b(j))\,,
	\quad \fT(\ell,v,b):= \sum_{ \ell_1+\ell_2=\ell} \fT_1(\ell_1,v,b)\  \fT_2 (\ell_2,v,b - v\cdot \pi(\ell_1))\,.
	\]
	We estimate its $|T|^{\mathtt{ T}}_{a, -m}$ norm, given by  \eqref{norm.toplitz}, with $m = \min(m_1, m_2)$:
	\begin{align*}
&	|T|^{\mathtt{ T}}_{a, -m}\leq \sup_{ v\in \cV\,, b\in\Z} \sum_{ \ell\in\Z^d } e^{a({\tc}|v|+|\ell|)}| \fT(\ell,v,b)| \jap{b}^{m}\\
		&	\leq  \sup_{ v\in \cV\,, b\in\Z} \sum_{ \ell_1\in\Z^d } 	\!\!\! e^{a|\ell_1|}| \fT_1(\ell_1,v,b)|\jap{b}^{m_1}
	\!\!\!	\!\!\!	\!	\sup_{ v'\in \cV\,, b'\in\Z} \sum_{ \ell_2\in\Z^d }	\!\!\! e^{a({\tc}|v'|+|\ell_2|)}| \fT_2(\ell_2,v',b')|
		 \le |M^\tT_1|_{a,-m_1}^{\mathtt{ T}} |M^\tT_2|_{a, 0}^{\mathtt{ T}}\,.
	\end{align*}
	Then we set
	\begin{align}\label{decomp.r}
		R :=
		M_1 M_2- T  = 	
		\sum_{i = 1}^{m_1} M^{(i)}_{1}  M^\tT_{2} + \sum_{i= 1}^{m_2}  M^\tT_1 M^{(i)}_{2}  + \sum_{i_1 = 1}^{m_1} \sum_{i_2 = 1}^{m_2} M^{(i_1)}_1 \,  M^{(i_2)}_{2} + (M^\tT_1 \,  M^\tT_{2} - T)\,.
	\end{align}
	We start with estimating the first summand. To each term in the sum we apply Lemma \ref{tigre.contro.tigre} with $R = M_1^{(i)}\in \cL_{\frac{a}{2}; -(i, m_1-i)}$ and $T = M_2^\tT\in \cT_{a, -m_2}$, deducing that 
	$$
	\begin{aligned}
		|M_1^{(i)}\,   M^\tT_2|_{\frac{a}{2}; -(i, m_1-i)} &
		{\lesssim a^{-p-2}}
		|M_1^{(i)}|_{\frac{a}{2}; -(i, {m_1-i})} \, 
		| M^\tT_2|^{\tT}_{a, \vec 0}
		{\stackrel{\eqref{Miota.est}} {\lesssim}} a^{-p-2} \sega{M_1}{a}{-m_1} \sega{M_2}{a}{-m_2}\,.
	\end{aligned}
	$$
	The second summand in \eqref{decomp.r} is estimated analogously using  Lemma \ref{tigre.contro.tigre}   with $R = M_2^{(i)} \in \cL_{\frac{a}{2}; -(i, m_2-i)}$ and $T=M^\tT_1 \in \cT_{a, -m_1}$, getting 
	$$
	\begin{aligned}
		|M^\tT_{1} \, M^{(i)}_{2}|_{\frac{a}{2}; -(i, m_1)} &\lesssim_{j, m_1}
		a^{-\mu i} |M^{(i)}_2|_{\frac{a}{2}, -(i, 0)} \, |M^\tT_{2}|^{\tT}_{a, -m_1}
		\stackrel{\!\!\!\!\!\!\!\!\!\!\!\!\!\eqref{Miota.est}}{\lesssim_{m_1, m_2}} a^{-\mu m_2 } \sega{M_1}{a}{-m_1} \sega{M_2}{a}{-m_2}\,.
	\end{aligned}
	$$
	Furthermore, by Lemma \ref{lemma.algebra} one has that $\forall j_1 = 1, \dots, m_1$ and $\forall j_2 = 1, \dots, m_2$
	$$
	\begin{aligned}
		|M^{(i_1)}_{1}\,  M^{(i_2)}_{2}|_{\frac{a}{2}; -(i_1, m_1-i_1)} &\leq
		|M^{(i_1)}_{1}|_{\frac{a}{2}; -(i_1, m_1-i_1)} \, |M^{(i_2)}_{2}|_{\frac{a}{2}; (0, 0)}
		\stackrel{\eqref{Miota.est}}{\leq} 4 \sega{M_1}{a}{-m_1} \sega{M_2}{a}{-m_2}\,.
	\end{aligned}
	$$
	Remark that, since $m = \min\{m_1, m_2\}$,  in particular $\forall i_1 = 1, \ldots, m_1$, $\forall i_2 = 1, \ldots, m_2$
	\begin{gather*}
		M^{(i_1)}_{1}  \, M^\tT_{2} \in \cL_{\frac{a}{2}; -(i_1, m -i_1)} \ ,  \quad 
		M^\tT_{1} \, M^{(i_2)}_{2} \in \cL_{\frac{a}{2}; -(i_2, m)} , 
		\quad
		M^{(i_1)}_{1}\,  M^{(i_2)}_{2} \in \cL_{\frac{a}{2}; -(i_1, m -i_1)}  \ . 
	\end{gather*}
	Thus it only remains to estimate the last summand in \eqref{decomp.r}. In particular, we are going to prove that $M^\tT_1 \,  M^\tT_{2} - T \in \cL_{\frac{a}{2}; -(m, 0)}$.\\
	Recalling that $j-j'= \pi(\ell)$, we have:
	\[
	(M^\tT_1 \,  M^\tT_{2} - T)_j^{j'}(\ell) =  \sum_{ \ell_1+\ell_2=\ell \atop j_1= j-\pi(\ell_1)} \fT_1(\ell_1,v(j),b(j))\Big(\fT_2(\ell_2,v(j_1),b(j_1))-  \fT_2(\ell_2,v(j),b(j) - v(j)\cdot \pi(\ell_1))\Big)
	\]
	Now if $v(j)= v(j_1)$ then 
	$b(j_1):= v(j_1)\cdot j_1 = v(j)\cdot (j-\pi(\ell_1)) = b(j) - v(j)\cdot \pi(\ell_1)$,
	and thus the term in the parenthesis above disappears. 
	Consequently we can restrict the sum to the indexes $j_1$ fulfilling 
	$v(j_1) \neq v(j)$.
	
	For these remaining terms, using that  $|j-j'|=|\pi(\ell)|\le \tc^{-1} |\ell|$, we set {$m'\leq m_1+m_2$} and estimate 
	\begin{align}
		\notag
		&\left(\und{M^\tT_1 \,  M^\tT_{2} - T}_{\, \frac{a}{2}; -(m', 0)}\right)_j^{j'}\le e^{-\tc \frac{a}{4} |j-j'|} \! \! \! \!
		\sum_{ \substack{\ell_1,\ell_2 \\ j_1= j-\pi(\ell_1),\\ v(j)\ne v(j_1) }} \! \! \! \! e^{\frac34 a|\ell_1+\ell_2|}
		\, |\fT_1(\ell_1,v(j),b(j))| \, |\fT_2(\ell_2,v(j_1),b(j_1))| \, 
		\jap{j}^{\mu m'} \\ &
		+ e^{-\tc \frac{a}{4} |j-j'|}  \sum_{ \substack{\ell_1,\ell_2 \\ j_1= j-\pi(\ell_1),\\ v(j)\ne v(j_1) }} e^{\frac34 a|\ell_1+\ell_2|}|\fT_1(\ell_1,v(j),b(j))| |\fT_2(\ell_2,v(j),b(j) - v(j)\cdot \pi(\ell_1))| \jap{j}^{\mu m'}\,.
		\label{est.chiavetta}
	\end{align}
	If we show that the two sums above are uniformly bounded, then
	Lemma \ref{gigetto} will gives us the needed boundedness.
	Let us start with the first sum. We distinguish two cases.
	
	\noindent  \underline{Case 1:} $\jap{j} \le \tJ_\delta$ (given by Lemma \ref{traviata}).  Then trivially
$$
	\sum_{ \substack{\ell_1,\ell_2 \\ j_1= j-\pi(\ell_1),\\ v(j)\ne v(j_1) }} \!\!\!\!\!\!
	e^{\frac34 a|\ell_1+\ell_2|}\, |\fT_1(\ell_1,v(j),b(j))|\, |\fT_2(\ell_2,v(j_1),b(j_1))| \,  \jap{j}^{\mu m'} \lesssim_\delta |M^\tT_{1}|^\tT_{a,0}
	\, |M^\tT_{2}|^\tT_{a,0}\,.
	$$
	\noindent  \underline{Case 2:}	 $\jap{j} >\tJ_\delta$. We further distinguish 4 subcases:\\
	(A) $|\ell_1|> \tc \jap{j}^\delta$;\ \ \ (B) $|\ell_1|\le  \tc \jap{j}^\delta$ and $|v(j_1)|> \frac12 \jap{j}^\delta$; \ \ \ (C) $|\ell_1|\le  \tc \jap{j}^\delta$,  $|v(j_1)|\le \frac12 \jap{j}^\delta$ and $|b(j)|> \jap{j}^\mu$;\ \ \ (D)  $|\ell_1|\le  \tc \jap{j}^\delta$,  $|v(j_1)|\le \frac12 \jap{j}^\delta$ and $|b(j)| \le  \jap{j}^\mu$.
	\vspace{7pt}\\
	In case (A) we have
	\begin{align*}
		&\sum_{ \substack{\ell_1,\ell_2 \\ j_1= j-\pi(\ell_1),\\ v(j)\ne v(j_1) }}^{(A)} e^{\frac34 {a}|\ell_1+\ell_2|} \, |\fT_1(\ell_1,v(j),b(j))|
		\, |\fT_2(\ell_2,v(j_1),b(j_1))|\, 
		\jap{j}^{\mu m'} \\
		&\lesssim_{\tc, \mu} \!\!\! \sum_{ \substack{\ell_1}}
		e^{\frac34{a}|\ell_1|} |\ell_1|^{\frac{\mu m'}{\delta}} |\fT_1(\ell_1,v(j),b(j))| |M^\tT_2|^\tT_{{a},0}
		 \lesssim_{\tc, \mu, \delta} {a^{-\frac{\mu m'}{\delta}}} |M^\tT_1|^\tT_{a,0} \, |M^\tT_2|^\tT_{{a},0}\,.
	\end{align*}
	In case (B) we get 
	\begin{align*}
		&\sum_{ \substack{\ell_1,\ell_2 \\ j_1= j-\pi(\ell_1),\\ v(j)\ne v(j_1) }}^{(B)} e^{\frac34{a}|\ell_1+\ell_2|}\, 
		|\fT_1(\ell_1,v(j),b(j))|
		\, 
		|\fT_2(\ell_2,v(j_1),b(j_1))| \, \jap{j}^{\mu m'} \\
		&\lesssim_{\mu, \delta}  |M_1^\tT|^\tT_{a,0}  \sum_{ \ell_2} e^{\frac34{a}|\ell_2|}|\fT_2(\ell_2,v(j_1),b(j_1))|
		\, |v(j_1)|^{\frac{\mu m'}{\delta}}\lesssim_{\mu, \delta, \tc} {a^{-\frac{\mu m'}{\delta}}} \, |M_1^\tT|^\tT_{a,0} \, |M_2^\tT|^\tT_{{a},0}\,.
	\end{align*}
	In case (C) we start by remarking that $|b(j_1)| > \frac12 \jap{j}^\mu$ provided $\tJ_\delta$ is large enough. Indeed assume by contradiction that $|b(j_1)| \le \frac12 \jap{j}^\mu$.
	Then, by the definition of $b(j)$, we have that
	\[
	|b(j)| \le 
	|v(j_1)\cdot j|\le |v(j_1)\cdot j_1|+ | v(j_1) \cdot \pi(\ell_1)| \le |b(j_1)| + \jap{j}^{2\delta} \le \frac12 \jap{j}^\mu + \jap{j}^{2\delta}\le \jap{j}^\mu\,,
	\]
	where 	in the last inequality we have used
	$\jap{j} > 2^{\frac1{1-4\delta}}$,   contradicting the hypotheses.
	Now, as $\mu \leq 1$, we bound
	\begin{align*}
		&\sum_{ \substack{\ell_1,\ell_2 \\ j_1= j-\pi(\ell_1),\\ v(j)\ne v(j_1) }}^{(C)} e^{\frac34{a}|\ell_1+\ell_2|}\, |\fT_1(\ell_1,v(j),b(j))|
		\, |\fT_2(\ell_2,v(j_1),b(j_1))| \, \jap{j}^{m_1 \mu}\,  \jap{j}^{m_2 \mu} \\
		&\lesssim \sum_{ \substack{\ell_1}}
		e^{\frac34{a}|\ell_1|}  |\fT_1(\ell_1,v(j),b(j))|\jap{b(j)}^{m_1} \sum_{\ell_2 } e^{\frac{a}{2}|\ell_2|}\,
		|\fT_2(\ell_2,v(j_1),b(j_1))|\, \jap{b(j_1)}^{m_2}\\
		& \lesssim_{\mu, \delta}
		a^{-\frac{\mu {(m_1 + m_2)}}{\delta}}\, 
		|M^\tT_2|^\tT_{a,-m_1}\, |M^\tT_{1}|^\tT_{a,-m_2}.
	\end{align*}
	Concerning  (D), we claim that in this case $v(j)=v(j_1)$, so  there is no constribution to the sum. 
	Indeed if $v(j_1) \neq v(j)$,  Lemma  \ref{traviata} $(ii)$ (which we can apply as  $|j-j_1| \leq \la j \ra^\delta$) would imply   $|v(j_1)|> \frac12 \la j \ra^\delta$, getting a contradiction.
	This concludes the proof of \eqref{stima.chiavetta}.\\ \vspace{7pt} \em Proof of \eqref{a.mano}:
	Let $M^\tT_i\in \cT_{a, -1}$ and $M_i^{(1)} \in \cL_{a; -(1, 0)}$ be quasi-T\"oplitz decompositions of the $M_i$.
		$$
		|M^\tT_i|^\tT_{a, -1} + |M^{(1)}_i|_{a; (-1, 0)} \leq (1 + \varepsilon) \sega{M_i}{a}{-1}\,.
		$$
		Arguing as in the proof of Proposition \ref{prop.chiavetta},
		to each $M_i^\tT$ we associate the symbol $\fT_i(\ell,v,b)$. We set
		\[
		\fT(\ell,v,b):= \sum_{ \ell_1+\ell_2=\ell} \fT_1(\ell_1,v,b)\fT_2(\ell_2,v,b - v\cdot \pi(\ell_1))\,,\quad 	(T)_j^{j'}(\ell):= \fT(\ell,v(j),b(j))\,,
		\]
		and we define the line-Toplitz part of $M_1 M_2$ as  $(M_1 M_2)^\tT=T$.
		Then $T$ is by definition line-T\"oplitz, moreover to estimate $|T|^{\tT}_{a', -2}$ $\forall a'< a$ we only need to control:
		\begin{align*}
			&  \sup_{ v\in \cV\,, b\in\Z}\sum_{ \ell\in\Z^d } e^{a'({\tc}|v|+|\ell|)}\left| \sum_{ \ell_1+\ell_2=\ell} \fT_1(\ell_1,v,b)\fT_2(\ell_2,v,b - v\cdot \pi(\ell_1))\right| \jap{b}^{2}\\
			&	\lesssim \sup_{ v\in \cV\,, b\in\Z} \sum_{ \ell_1\in\Z^d } e^{a' ({\tc}|v| + |\ell_1|)}| \fT_1(\ell_1,v,b)|\jap{b}^{}\jap{v\cdot \pi(\ell_1)}^{}
			\sup_{ v'\in \cV\,, b'\in\Z} \sum_{ \ell_2\in\Z^d }e^{a'({\tc}|v'| + |\ell_2|)}| \fT_2(\ell_2,v',b')|\jap{b'}^{}
			\\
			& \lesssim  \sup_{ v\in \cV\,, b\in\Z} \sum_{ \ell_1\in\Z^d } e^{a(|\ell_1|+{\tc}|v|)}| \fT_1(\ell_1,v,b)|\jap{b}^{}e^{-(a - a'){\tc}|v| -(a-a')|\ell_1|}\jap{v} \jap{\ell_1}^{}
			|M^\tT_2|_{a',-1}^{\mathtt{ T}}
			\\ & \lesssim_{{\tc}} (a-a')^{-2}|M^\tT_1|_{a,-1}^{\mathtt{ T}} |M^\tT_2|_{a',-1}^{\mathtt{ T}}\,.
		\end{align*}
		We are left with estimating
		\begin{equation}\label{resti}
			M_1 M_2 - T = M^\tT_1 M^{(1)}_2 + M^{(1)}_1 M^\tT_2 + M^{(1)}_1 M^{(1)}_2 + (M^\tT_1 M^\tT_2  - T)\,.
		\end{equation}
		We will actually prove the following:
		$$
		M^\tT_1 M^{(1)}_2  \in \cL_{\frac{a'}2; -(1,1)}\,, \quad M^{(1)}_1 M^\tT_2 \in \cL_{\frac{a'}2; -(1,1)} + \cL_{\frac{a'}2; -(2,0)}\,, \quad M^{(1)}_1 M^{(1)}_2 + (M^\tT_1 M^\tT_2  - T) \in \cL_{\frac{a'}2; -(2,0)}\,.
		$$
		The estimate of the last two terms in \eqref{resti} is straightforward: by Formula \eqref{stima2} in Lemma \ref{lemma.algebra} one has $M^{(1)}_1 M^{(1)}_2 \in \cL_{\frac{a'}2; -(2, 0)},$ with
		$$
		|M^{(1)}_1 M^{(1)}_2|_{\frac{a'}2; {-(2, 0)}} \lesssim_{\mu, \delta}(a - a')^{-\frac{2\mu}{\delta}} |M^{(1)}_1|_{\frac{a}2;{-(1,0)}} |M^{(1)}_2|_{\frac{a}2;{-(1,0)}}\,, 
		$$
		while arguing as in the proof of Proposition \ref{prop.chiavetta}, formula \eqref{decomp.r} with $m'=2$, one obtains that there exists $\gamma = \gamma(\mu, \delta, p)>0$ such that $T_1 T_2 - T \in \cL_{\frac{a'}2; -(2, 0)},$ with 
		$$
		|M^\tT_1 M^\tT_2 - T|_{\frac{a'}2; (-2, 0)} \lesssim_{\mu, \delta, \tc, p} a^{-\gamma} |M^\tT_1|^\tT_{a, -1} |M^\tT_2|^\tT_{a, -1}\,.
		$$
		Concerning $M^\tT_1 M^{(1)}_2 $, one has:
		\begin{align*}
			\hspace{-3pt}&(\underline{M^\tT_1 M^{(1)}_2 }_{\frac{a'}2; -(1, 1)})_{j}^{j'}
			\lesssim_{\mu} \hspace{-15pt}\sum_{ \ell_1, \ell_2: \atop  j-j'= \pi(\ell_1+\ell_2) } \hspace{-12pt} e^{\frac{a'}2|\ell_1+\ell_2|} |(M^\tT_1)_{j}^{j-\pi(\ell_1)}(\ell_1)|\jap{j-\pi(\ell_1)}^{\mu} \jap{\pi(\ell_1)}^{\mu} \langle b(j)\rangle |(M^{(1)}_2)_{j-\pi(\ell_1)}^{j'}(\ell_2) |\\
			&\lesssim_{\mu} a^{-\mu} 
			\sum_{j_1} ((\underline{M^\tT_1})_{\frac 34 a;-(0,1) })_{j}^{j_1} ((\underline{M^{(1)}_2})_{\frac{a'}2; -(1, 0)})_{j_1}^{j'}\,,
		\end{align*}
		which  implies
		$$
		|M^\tT_1 M^{(1)}_2|_{\frac{a'}2; -(1,1)}\lesssim_{\mu} a^{-\mu} |M^\tT_1|^\tT_{a, -1} |M^{(1)}_2|_{\frac{a'}2; -(1, 0)}\,,
		$$
		{by Lemma \ref{lemma.algebra} and by Lemma \ref{vaccab}}.\\
		In order to conclude the proof, we now estimate $M^{(1)}_1 M^\tT_2$. Recalling that one has
		\begin{equation}\label{a.sum}
			(M^{(1)}_1 M^\tT_2)_{j}^{j'}(\ell) =  \sum_{ \ell_1, \ell_2: \ell_1 + \ell_2 = \ell} (M^{(1)}_1)_{j}^{j-\pi(\ell_1)}(\ell_1) (M^\tT_2)_{j-\pi(\ell_1)}^{j'}(\ell_2) \,,
		\end{equation}
		with $(M^{(1)}_1 M^\tT_2)_j^{j'} \neq 0$ only if $j - j' = \pi(\ell)$, we decompose $M^{(1)}_1 M^\tT_2$ as follows:
		\begin{equation}\label{another.sum}
			M^{(1)}_1 M^\tT_2 = (M^{(1)}_1 M^\tT_2)^{(A)} + (M^{(1)}_1 M^\tT_2)^{(B)} + (M^{(1)}_1 M^\tT_2)^{(C)} + (M^{(1)}_1 M^\tT_2)^{(D)} + (M^{(1)}_1 M^\tT_2)^{(E)}\,,  
		\end{equation}
		where the summands in \eqref{another.sum} respectively corresponds to the cases:
		\begin{enumerate}
			\item[(A)] $\jap{j - \pi(\ell_1)} > \tJ_\delta$, with $\tJ_\delta$ given by Lemma \ref{traviata}, and $|\ell_1| > \tc \jap{j}^\delta$;
			\item[(B)] $\jap{j-\pi(\ell_1)} > \tJ_\delta$, $|\ell_1| \leq \tc \jap{j}^\delta$ and $|v(j - \pi(\ell_1))| > \frac{1}{2} \jap{j - \pi(\ell_1)}^\delta$;
			\item[(C)] $\jap{j-\pi(\ell_1)} > \tJ_\delta$, $|\ell_1| \leq \tc \jap{j}^\delta$, $|v(j - \pi(\ell_1))| \leq \frac{1}{2} \jap{j - \pi(\ell_1)}^\delta$, and $|b(j - \pi(\ell_1))| > \jap{j- \pi(\ell_1)}^\mu$;
			\item[(D)] $\jap{j-\pi(\ell_1)} > \tJ_\delta$, $|\ell_1| \leq \tc \jap{j}^\delta$, $|v(j -\pi(\ell_1))| \leq \frac{1}{2} \jap{j - \pi(\ell_1)}^\delta$ and $|b(j - \pi(\ell_1))| \leq \jap{j - \pi(\ell_1)}^\mu$;
			\item[(E)] $\jap{j-\pi(\ell_1)} \leq \tJ_\delta$.
		\end{enumerate}
		Concerning $(M^{(1)}_1 M^\tT_2)^{(A)}$,
			using $|\ell_1| > \tc \jap{j}^\delta$ one obtains
			$$
			((\underline{M^{(1)}_1 M^\tT_2})^{(A)}_{\frac{a'}2; -(2, 0)})_{j}^{j'} \lesssim_{\tc, \mu, \delta} (a-a')^{-\frac{2\mu}{\delta}} \sum_{j_1}((\underline{M^{(1)}_1})_{\frac{a}2, 0, 0})_{j}^{j_1} ((\underline{M^\tT_2})_{\frac{a}2, 0, 0})_{j_1}^{j'}\,.
			$$
		For $(M^{(1)}_1 M^\tT_2)^{(B)}$, we define $\theta := \tc^{-1} (a - a')/2$ and we 
		observe that
		$$
		\jap{j}^\delta \lesssim_\delta \jap{j - \pi(\ell_1)}^\delta + \jap{\pi(\ell_1)}^\delta \lesssim_\delta \jap{j - \pi(\ell_1)}^\delta \lesssim_\delta |v(j -\pi(\ell_1))|\,,
		$$
		from which one obtains
		$$
		((\underline{M^{(1)}_1 M^\tT_2})^{(B)}_{\frac{a'}2; -(2, 0)})_{j}^{j'} \lesssim_{\tc, \mu, \delta} {a^{-\frac{2\mu}{\delta}}}e^{-\theta |j - j'|} |M^{(1)}_1|_{\frac{a}2; (0, 0)} |M^\tT_2|^\tT_{\frac{{a}}2, 0}\,.
		$$
		Then the estimate on $(M^{(1)}_1 M^\tT_2)^{(B)}$ follows by Lemma \ref{gigetto} .
		The estimate of $(M^{(1)}_1 M^\tT_2)^{(C)}$ simply follows observing that, since
		$$
		\jap{j}^\mu \lesssim_\mu \jap{j - \pi(\ell_1)}^\mu + \jap{\pi(\ell_1)}^\mu \lesssim_{\mu} \jap{j -\pi(\ell_1)}^\mu \lesssim_{\mu} |b(j - \pi(\ell_1))|\,,
		$$
		one has
		\begin{align*}
			((\underline{M^{(1)}_1 M^\tT_2})^{(C)}_{\frac{a'}2; -(2, 0)})_{j}^{j'} &\leq \sum_{ \ell_1, \ell_2:   j-j'= \pi(\ell_1+\ell_2)\,, \atop |b(j - \pi(\ell_1))| \geq \jap{j}^\mu } e^{a'|\ell_1+\ell_2|} \jap{j}^{2\mu} |(M^{(1)}_1)_{j}^{j-\pi(\ell_1)}(\ell_1) (M^\tT_2)_{j-\pi(\ell_1)}^{j'}(\ell_2) |\\
			& \leq \sum_{ \ell_1, \ell_2:   j-j'= \pi(\ell_1+\ell_2) } e^{\frac{a'}2|\ell_1+\ell_2|} \jap{j}^{\mu} \jap{b(j - \pi(\ell_1))} |(M^{(1)}_1)_{j}^{j-\pi(\ell_1)}(\ell_1) (M^\tT_2)_{j-\pi(\ell_1)}^{j'}(\ell_2) |\\
			& \leq \sum_{j_1} ((\underline{M^{(1)}_1})_{\frac{a'}2, -(1, 0)})_{j}^{j_1} ((\underline{M^\tT_2})_{\frac{a'}2, -(0, 1)})_{j}^{j_1}\,,
		\end{align*}
		{and applying Lemma \ref{lemma.algebra} and Lemma \ref{vaccab}}.
		In order to estimate $(M^{(1)}_1 M^\tT_2)^{(D)}$, we start with recalling that, by Item 2 of Lemma \ref{traviata} with $j$ replaced by $j - \pi(\ell_1)$ and $h$ replaced by $j$, one has $v(j - \pi(\ell_1)) = v(j)$. Then one has
		\begin{align*}
			{|b(j)|} &= {|j \cdot v(j - \pi(\ell_1))|}
			 \leq {|(j - \pi(\ell_1)) \cdot v(j - \pi(\ell))| + |\pi(\ell_1)| |v(j - \pi(\ell_1)|} \\
			&\leq \jap{b(j - \pi(\ell_1))} + \tc^{-1} |\ell_1| |v(j-\pi(\ell_1))|\\
			&\lesssim_{\tc} \jap{b(j - \pi(\ell_1))} +\jap{\ell_1} \jap{v(j-\pi(\ell_1))} \,.
		\end{align*}
		This in turn, setting $\sigma :=  (a - a')/2$ and $\theta := \tc^{-1} \s$, enables to deduce
		\begin{align*}
		&((\underline{M^{(1)}_1 M^\tT_2})^{(D)}_{\frac{a'}2; -(1, 1)})_{j}^{j'}\\
		&\lesssim_{\tc} \sum_{ \ell_1, \ell_2: \atop  j-j'= \pi(\ell_1+\ell_2) } e^{\frac{a'}2|\ell_1+\ell_2|} |(M^{(1)}_1)_{j}^{j-\pi(\ell_1)}(\ell_1)|\jap{j}^{\mu} \jap{b(j - \pi(\ell_1))} |\ell_1| |v(j - \pi(\ell_1))| (M^\tT_2)_{j-\pi(\ell_1)}^{j'}(\ell_2) |\\
		&\lesssim_{\tc} (a \sigma)^{-1} e^{-\frac{\theta}2 |j - j'|} |M^{(1)}_1|_{\frac a2; -(1, 0)} |M^\tT_2|^\tT_{a, -1}\,.
		\end{align*}
		Finally, using that $\jap{j}^{2\mu} \lesssim_\mu \jap{j - \pi(\ell_1)}^{2\mu} + \jap{\pi(\ell_1)}^{2\mu} \lesssim{\mu, \delta,\tc} \tJ_\delta^{2\mu}|\ell_1|^{2\mu}\,,$
		one obtains
		$$
		((\underline{M^{(1)}_1 M^\tT_2})^{(E)}_{\frac{a'}2; -(2, 0)})_{j}^{j'} \lesssim_{\delta, \tc, \mu} (a - a')^{-2\mu} \sum_{j_1} ((\underline{M^{(1)}_1})_{\frac{a}2; (0, 0)})_{j}^{j_1} ((\underline{M^\tT_2})_{\frac{a'}2; (0, 0)})_{j_1}^{j'}\,.
		$$
	\end{proof}
\begin{proof}[Proof of Lemma \ref{commuta.bene}]
	We set (recall that $V(\ell)=0$ if $\pi(\ell)=0$) $(S_0^{\tT})_{j}^{j'}(\ell)= \fS(\ell, v(j),b(j))$ with
\begin{equation}
\label{porcheria}
\fS(\ell, v,b)= \begin{cases}
\dfrac{- V(\ell)}{ \omega\cdot \ell + 2\frac{|\pi(\ell)|}{|v|} b - |\pi(\ell)|^2}\,,\quad {\rm if} \quad \pi(\ell)\ne 0\,,\quad  v \parallel \pi(\ell)\\0 \qquad {\rm otherwise}
\end{cases}\,.
\end{equation}
One has
\[ 
\begin{aligned}
|S_0^{\tT}|^{\mathtt{T}}_{a',-1} &{\leq} \sup_{v\in \cV\,, b\in\Z}\sum_{ \ell \in \Z^d} {e^{a'\tc |v| + {a'}|\ell|}}  |\mathfrak S(\ell,v,b) | \jap{b}^{1}\\
&= \sup_{v\in \cV\,, b\in\Z}\sum_{ \ell \in \Z^d}  {e^{a'\tc |v| + {a'}|\ell|}} \frac{|V(\ell)|}{| \omega\cdot \ell + 2\frac{|\pi(\ell)|}{|v|} b - |\pi(\ell)|^2|} \jap{b}^{1}\,.
\end{aligned}
\]
We estimate $*=\dfrac{\jap{b}}{ |\omega\cdot \ell  - |\pi(\ell)|^2 + 2 \frac{|\pi(\ell)|}{|v|} b|}$  by distinguishing 
two cases.
If $|b|> |\omega\cdot \ell  - |\pi(\ell)|^2|$  then (recalling that $\frac{|\pi(\ell)|}{|v|}$ is a non-zero integer) one has $*<2$.
Otherwise if $|b|\leq |\omega\cdot \ell  - |\pi(\ell)|^2|$ then, {by \eqref{diofantino} and recalling that $|\pi(\ell)|\le \tc^{-1}|\ell|$},
\[
* \le  (|\omega\cdot \ell  - |\pi(\ell)|^2|+1) |\ell|^\tau \gamma^{-1} \lesssim \g^{-1}|\ell|^{\tau+2}\,.
\]
In conclusion, 
\begin{equation}
\begin{aligned}\label{top}
|S_0^{\tT}|^{\mathtt{T}}_{a',-1} 	&\lesssim \g^{-1}\sum_{ \ell: j-j'= \pi(\ell)\ne 0} {e^{a' \tc |\pi(\ell)| +{a'}|\ell|}} |V(\ell)| |\ell|^{(\tau+2)}\\
&\lesssim \g^{-1}\sum_{ \ell: j-j'= \pi(\ell)\ne 0} {e^{{2 a'}|\ell|}} |V(\ell)| |\ell|^{(\tau+2)}\lesssim_{\tc, \gamma} {(\ta - 2 a')^{- {(\tau+2)}}} \|V\|_{\ta, \tp}\,.
\end{aligned}
\end{equation}
Recalling that, if  $\pi(\ell)\parallel v(j)$, $ j \cdot \pi(\ell) = \frac{| \pi(\ell)|}{|v(j)|}b(j),$ we have 
\[
(S_0-S_0^{\tT})_j^{j'}(\ell)=
\begin{cases}
\dfrac{- V(\ell)}{ \omega\cdot \ell + 2 j \cdot \pi(\ell) - |\pi(\ell)|^2}\,,\quad &j'= j-\pi(\ell)\,,\;\pi(\ell)\ne 0\,,\; \pi(\ell)\not\parallel v(j)\,,\\
0 \quad &\textnormal{otherwise}\,.
\end{cases}
\]
Accordingly, we divide the  indexes $j,\ell$ which contribute to $S_0-S_0^\tT$   following the two conditions:
\\
i)\ One either has $|\ell|\ge \tc \jap{j}^\delta$ or $
\jap{j} \le \tj_{\rm min}:= 2^{\frac1{2\delta}}\,
$; \ ii) One has $|\ell| < \tc \jap{j}^\delta$, $\jap{j} > \tj_{\rm min}$  and $\pi(\ell)\not\parallel v(j)$.
We decompose accordingly $S_0-S_0^{\tT}= R + P$.
\\
First, by applying Lemma \ref{Bony} to $S_0-S_0^{\tT}$ and setting $R_1= (S_0-S_0^{\tT})^{(R)}$, we deduce
\[
{|R_1|_{\frac{a'}{2};-(2, 0)} \lesssim_{\delta} \ta^{-\frac{2 \mu}{\delta}} \|V\|_{\ta, \tp} \,.}
\]
To deal with the case $|j| \leq j_{\rm min}$, we set
$
(R_2)_{j}^{j'}(\ell):= (S_0)_{j}^{j'}(\ell)$ if $ j \leq\tj_{\rm min}\,.
$
\\
Since if $|j| \leq j_{\rm min}$ we have
\begin{equation}
\frac{\jap{j}^{2\mu }}{|\omega\cdot \ell + 2 j \cdot \pi(\ell) - |\pi(\ell)|^2|} \le   \g^{-1}\jap{\tj_{\rm min}}^{2\mu }|\ell|^{\tau}\,,
\end{equation}
defining $R := R_1 + R_2,$ we get that
\begin{equation}\label{smoo}
{|R|_{\frac{a'}{2}; -(2, 0)} \lesssim_{\delta, \tau} \ta^{-\max\{\tau, \frac{2 \mu}{\delta}\}} \|V\|_{\ta, \tp} \,.}
\end{equation}
To deal with case ii), we define 
\[
{P }_j^{j'}(\ell):= \begin{cases}
\dfrac{- V(\ell)}{ \omega\cdot \ell + 2 j \cdot \pi(\ell) - |\pi(\ell)|^2} & \mbox{if }\; j>\tj_{\rm min}, \quad |\ell|< \tc\jap{j}^\delta \quad \mbox{and }\; \pi(\ell)\not\parallel v(j)\,,\\
0 & \mbox{otherwise}
\end{cases}    \,.
\]
Now we claim that if $\jap{j}\ge 2^{\frac1{2\delta}}$, $|\ell| < \tc \jap{j}^\delta$  and $\pi(\ell)\not\parallel v(j)$, then
$|j\cdot \pi(\ell)|>\jap{j}^{\mu}$.
\\
Indeed, $|\ell| < \tc \jap{j}^\delta$ implies that $|\pi(\ell)|< \jap{j}^\delta$.  Fixing $w\in \cV$ to be the  direction parallel to $\pi(\ell)$ one has $|w|<  \jap{j}^\delta$ hence
\[
|\pi(\ell)\cdot j|\ge |w\cdot j|\ge |b(j)|:= \min_{v\in\cV: |v|< \jap{j}^\delta }|j\cdot v|
\] Then, if $|b(j)|>\jap{j}^{\mu} $ our claim follows. Otherwise, we prove the claim by contradiction. Recall that $v(j)$ attains the minimum above and by ii) $v(j)\not\parallel w$.
If $|w\cdot j|\le \jap{j}^{\mu}$ then, by Lemma \ref{lemma:cramer} 
$
\jap{j}< \sqrt{2}\jap{j}^{1-\delta}
$,
which contradicts $\jap{j}\ge 2^{\frac1{2\delta}}$.
\\
As a consequence if $\jap{j}\ge 2^{\frac1{2\delta}}$, $|\ell| < \tc \jap{j}^\delta$  and $\pi(\ell)\not\parallel v(j)$, then
\begin{align*}
&\frac{\jap{j}^{\mu}}{|\omega\cdot \ell + 2 j \cdot \pi(\ell) - |\pi(\ell)|^2|} \le  \frac{\jap{j}^{\mu}}{2 |j \cdot \pi(\ell)|-  |\omega\cdot \ell  - |\pi(\ell)|^2|} 
\le \frac{\jap{j}^{\mu}}{2 \jap{j}^{\mu}-  (|\omega|+\mathtt c^{-2}) \tc^2\jap{j}^{2\delta}}\le 1
\end{align*}
provided that  $\jap{j} \ge \tj_{\rm min}$.
Therefore we have
\begin{equation}\label{nonpseudo}
{|P|_{\frac{a'}{2}; -(1, 0)} \lesssim \|V\|_{\ta, \tp}\,.}
\end{equation}
Recalling that $S_0 = S_0^{\tT} + P+R$,  estimates \eqref{top}, \eqref{smoo}, \eqref{nonpseudo} then imply that $S$ is quasi- T\"oplitz of order $-1$.
We now prove estimate \eqref{comm.hom}.
We start with explicitly computing $[S_0, M_V]$. Setting $h = j - \pi(\ell)$, one has
\begin{align*}
&([S_0,M_V])_j^{h}(\ell) = \sum_{ \ell_1+\ell_2=\ell} V(\ell_1) \big((S_0)_{j}^{j-\pi(\ell_2)}(\ell_2) -(S_0)_{j-\pi(\ell_1)}^{j-\pi(\ell)}(\ell_2) \big)\\
&= \sum_{ \ell_1+\ell_2=\ell} V(\ell_1) V(\ell_2)\left( \frac{1}{\omega\cdot \ell_2 + 2 (j-\pi(\ell_1)) \cdot \pi(\ell_2)  - |\pi(\ell_2)|^2} - \frac{1}{\omega\cdot \ell_2 + 2 j \cdot \pi(\ell_2) - |\pi(\ell_2)|^2}\right)\\
&=  \sum_{ \ell_1+\ell_2=\ell} \frac{-2\left(\pi(\ell_1)\cdot\pi(\ell_2)\right)V(\ell_1) V(\ell_2)}{(\omega\cdot \ell_2 + 2 j \cdot \pi(\ell_2) - |\pi(\ell_2)|^2)(\omega\cdot \ell_2 + 2 (j-\pi(\ell_1)) \cdot \pi(\ell_2)  - |\pi(\ell_2)|^2)}\,.
\end{align*}
We then split the above sum into two terms, according to whether $\pi(\ell_2) \parallel v(j)$ or not. In the first case, we set
\begin{equation}\label{top.part}
\begin{aligned}
&\hspace{-7pt}([S_0, M_V]^\tT)_{j}^{j - \pi(\ell)}(\ell)  := \!\!\!\!\sum_{ \ell_1+\ell_2=\ell \atop \pi(\ell_2) \parallel v(j)}  \frac{-\left(2\pi(\ell_1)\cdot\pi(\ell_2)\right) V(\ell_1) V(\ell_2)}{(\omega\cdot \ell_2 + 2 j \cdot \pi(\ell_2) - |\pi(\ell_2)|^2)(\omega\cdot \ell_2 + 2 (j-\pi(\ell_1)) \cdot \pi(\ell_2)  - |\pi(\ell_2)|^2)}\\
&= \sum_{ \ell_1+\ell_2=\ell \atop \pi(\ell_2) \parallel v(j)} \frac{-\left(2\pi(\ell_1)\cdot\pi(\ell_2)\right)  V(\ell_1) V(\ell_2)}{(\omega\cdot \ell_2 + 2 b(j)\frac{|\pi(\ell_2)|}{|v(j)|} - |\pi(\ell_2)|^2)(\omega\cdot \ell_2 + 2 b(j)\frac{|\pi(\ell_2)|}{|v(j)|} - 2 \pi(\ell_1)\cdot \pi(\ell_2) - |\pi(\ell_2)|^2)}\,,
\end{aligned}
\end{equation}
and we are going to prove that its symbol
$$
\fT(\ell, b, v):= \sum_{ \ell_1+\ell_2=\ell \atop \pi(\ell_2) \parallel v} \frac{-\left(2\pi(\ell_1)\cdot\pi(\ell_2)\right)  V(\ell_1) V(\ell_2)}{(\omega\cdot \ell_2 + 2 b(j)\frac{|\pi(\ell_2)|}{|v|} - |\pi(\ell_2)|^2)(\omega\cdot \ell_2 + 2 b\frac{|\pi(\ell_2)|}{|v|} - 2 \pi(\ell_1)\cdot \pi(\ell_2) - |\pi(\ell_2)|^2)} 
$$
is of order $-2$. In the second case, we set
\begin{equation}\label{rem.part}
R_j^h(\ell):=  \sum_{ \ell_1+\ell_2=\ell \atop \pi(\ell_2) \nparallel v(j)}  \frac{-\left(2\pi(\ell_1)\cdot\pi(\ell_2)\right) V(\ell_1) V(\ell_2)}{(\omega\cdot \ell_2 + 2 j \cdot \pi(\ell_2) - |\pi(\ell_2)|^2)(\omega\cdot \ell_2 + 2 (j-\pi(\ell_1)) \cdot \pi(\ell_2)  - |\pi(\ell_2)|^2)}\,,
\end{equation}
with $h = j - \pi(\ell),$ and we shall prove that $R \in {\cL}_{\frac{a'}{2}; -(2, 0)}\,.$\\
We start with {exhibiting estimates on $|[S_0, M_V]^\tT|^\tT_{a', -2}$.}
Let
\begin{equation}\label{fnonconta}
f^\epsilon(\ell_1, \ell_2):= \omega \cdot \ell_2 - |\pi(\ell_2)|^2 - 2\epsilon \pi(\ell_1)\cdot \pi(\ell_2)\,, \quad \epsilon \in \{0, 1\}\,,
\end{equation}
{so that the denominators appearing in \eqref{rem.part} have the form
	$$
	(2 j \cdot \pi(\ell_2) + f^0(\ell_1, \ell_2)) (2 j \cdot \pi(\ell_2) + f^1(\ell_1, \ell_2))\,;
	$$
}
arguing as to obtain estimates \eqref{top}, we observe that $\forall \epsilon \in \{0, 1\}$, if $\langle b\rangle \geq |f^\epsilon(\ell_1, \ell_2)|\,,$ one has
\begin{equation}\label{small.den.1}
\left|\frac{2 \langle b \rangle }{(\omega\cdot \ell_2 + 2 b(j)\frac{|\pi(\ell_2)|}{|v(j)|} - 2 \epsilon \pi(\ell_1)\cdot \pi(\ell_2) - |\pi(\ell_2)|^2)}\right| \leq 2\,,
\end{equation}
while if $\langle b \rangle < |f^\epsilon(\ell_1, \ell_2)|\,,$ one has
\begin{equation}\label{small.den.2}
\begin{aligned}
&\left|\frac{2 \langle b \rangle }{(\omega\cdot \ell_2 + 2 b(j)\frac{|\pi(\ell_2)|}{|v(j)|} - 2 \epsilon \pi(\ell_1)\cdot \pi(\ell_2) - |\pi(\ell_2)|^2)}\right| \leq \\ &\frac{2 |f^\epsilon(\ell_1, \ell_2)| }{\left|(\omega\cdot \ell_2 + 2 b(j)\frac{|\pi(\ell_2)|}{|v(j)|} - 2 \epsilon \pi(\ell_1)\cdot \pi(\ell_2) - |\pi(\ell_2)|^2)\right|}
 \lesssim \frac{|\ell_2|^2 |\ell_1|}{\gamma |\ell_2|^{-\tau}} \lesssim \gamma^{-1}|\ell_2|^{\tau + 2} |\ell_1|\,.
\end{aligned}
\end{equation}
{Let now $\theta := (\ta - 2 a')/2$.} Combining \eqref{small.den.1} and \eqref{small.den.2} {and recalling that, since $\pi(\ell_2)\parallel v$ and $v$ is a generator, one has $\tc |v| \leq \tc |\pi(\ell_2)| \leq |\ell_2|$}, one obtains
$$
\begin{aligned}
&{|[S_0, M_V]^\tT|}^\tT_{a', -2} \lesssim \sum_{\ell}  e^{{a'\tc |v| + {a'}|\ell|}} \sum_{ \ell_1+\ell_2=\ell \atop \pi(\ell_2) \parallel v} |V(\ell_1) V(\ell_2)| |\ell_2|^{2 (\tau + 2) +1} |\ell_1|^3\\
& \lesssim \sum_{\ell} {e^{{a'}|\ell|}} \sum_{ \ell_1+\ell_2=\ell \atop \pi(\ell_2) \parallel v}  {e^{a' |\ell_2|}} |V(\ell_1) V(\ell_2)| |\ell_2|^{2 (\tau + 2) +1}  |\ell_1|^3 \\ & \lesssim \sum_{\ell_1, \ell_2} {e^{{a'} |\ell_1| +{2 a'}|\ell_2|} }
|\ell_2|^{2 (\tau + 2) +1}  |\ell_1|^3 |V(\ell_1)||V(\ell_2)| \lesssim_{\tc, \tau} {(\ta - 2a')}^{-2(\tau + 2) - 4} \|V\|_{\ta, \tp}^2\,,
\end{aligned}
$$
provided $2a' \leq \ta$.\\
We now turn to prove that $R$ defined as in \eqref{rem.part} belongs to $\cL_{\frac{a'}{2}, -(2,0)}$. Letting $j_{\rm min} := 2^{\frac{1}{2\delta}}$, we decompose $R = R_1 + R_2 + R_3$ as follows:
\begin{gather*}
(R_1)_{j}^{h}(\ell) = R_{j}^{h}(\ell) \quad \textrm{ if } \max\{|\ell|, |\ell_1|, |\ell_2|\} > \tc \langle j \rangle^\delta\,, \quad (R_1)_{j}^{h}(\ell) = 0 \quad \textrm{otherwise}\,;\\
(R_2)_{j}^{h}(\ell) = R_{j}^{h}(\ell) \quad \textrm{ if }\jap{j} \leq j_{\min}\,, \quad (R_2)_{j}^{h}(\ell) = 0 \quad \textrm{ otherwise}\,;\\
(R_3)_{j}^{h}(\ell) = R_{j}^{h}(\ell) \quad \textrm{ if } \max\{|\ell|, |\ell_1|, |\ell_2|\} \leq \tc \langle j \rangle^\delta \quad \textrm{ and } \jap{j} \geq j_{\min}\,, \quad (R_3)_{j}^{h}(\ell) = 0 \quad \textrm{otherwise}\,.
\end{gather*}
Arguing as to prove estimate \eqref{smoo}, one easily sees that $R_1, R_2 \in \cL_{\frac{a'}{2}, -(N, 0)}$ for any $N \in \N$. We now prove that $R_3 \in \cL_{\frac{a'}{2},-(2, 0)}$. \\
First of all, we claim the following:
\begin{equation}\label{isbig}
(R_3)_{j}^h(\ell) \neq 0 \quad \Longrightarrow \quad |j\cdot \pi(\ell_2)|>\jap{j}^{\mu}\,. 
\end{equation}
Indeed, if $(R_3)_{j}^h(\ell) \neq 0,$ then by definition of $R_3$ and $R$ one has
$\jap{j} \ge j_{\min},$ $|\ell_2| < \tc \jap{j}^\delta$ and $\pi(\ell_2)\not\parallel v(j)$.
But then $|\ell_2| < \tc \jap{j}^\delta$ implies that $|\pi(\ell_2)|< \jap{j}^\delta$.  Fixing $w\in \cV$ to be the  direction parallel to $\pi(\ell_2)$ one has $|w|<  \jap{j}^\delta$, hence
\[
|\pi(\ell_2)\cdot j|\ge |w\cdot j|\ge |b(j)|:= \min_{v\in\cV: |v|< \jap{j}^\delta }|j\cdot v|
\] Then, if $|b(j)|>\jap{j}^{\mu} $ our claim follows. Otherwise, we prove the claim by contradiction. Recall that $v(j)$ attains the minimum above and, since $\pi(\ell_2)\nparallel v(j)$, $v(j)\not\parallel w$.
If $|w\cdot j|\le \jap{j}^{\mu}$ then, by Lemma \ref{lemma:cramer} 
\[
\jap{j}< \sqrt{2}\jap{j}^{1-\delta}
\]
which contradicts $\jap{j}\ge 2^{\frac1{2\delta}}$. This proves that \eqref{isbig} holds.
\\
As a consequence, taking $f^\epsilon(\ell_1, \ell_2)$ as in \eqref{fnonconta}, one has that, if $(R_3)_{j}^j(\ell) \neq 0$, then $\forall \epsilon \in \{0, 1\}$
\begin{align*}
&\frac{\jap{j}^{\mu}}{|\omega\cdot \ell_2 + 2 j \cdot \pi(\ell_2) - \epsilon 2 \pi(\ell_1)\cdot \pi(\ell_2) - |\pi(\ell_2)|^2|} = 	\frac{\jap{j}^{\mu}}{| 2 j \cdot \pi(\ell_2) + f^\epsilon(\ell_1, \ell_2)|} \\
& \le \frac{\jap{j}^{\mu}}{2 \jap{j}^{\mu}-  |\omega||\ell_2|-\mathtt c^{-2}|\ell_2|^2 - 2\mathtt{c}^{-2} |\ell_1| |\ell_2|} \le \frac{\jap{j}^{\mu}}{2 \jap{j}^{\mu}-  (|\omega|+2 \mathtt c^{-2}) \tc^2\jap{j}^{2\delta}}\le 1\,,
\end{align*}
since $\langle j \rangle \geq j_{\rm min}.$
This proves that
$$
\begin{aligned}
&\left|((\underline{R_3})_{\frac{a'}{2}; -(2,0)})_{j}^{h}\right|  := \sum_{ \ell: j-h= \pi(\ell)} e^{\frac{a'}{2}|\ell|} |(R_3)_{j}^{j'}(\ell)|\jap{j}^{2\mu} \\
& \lesssim_\tc \sum_{ \ell: j-h= \pi(\ell)} e^{\frac{a'}{2}|\ell|} \sum_{ \ell_1+\ell_2=\ell \atop \pi(\ell_2) \nparallel v(j)} |V(\ell_1) V(\ell_2)| |\pi(\ell_1) | |\pi(\ell_2)| \lesssim_\tc \ta^{-2} e^{-(\ta -\frac{a'}{2})|j-h|} \|V\|^2_{\ta, \tp}\,,
\end{aligned}
$$
and thus estimate \eqref{comm.hom} follows.
\end{proof}

\section{Covariant properties}
\label{section:cov}
Assume that $H$ is an Hamiltonian function which is invariant by translation and gauge:
\begin{equation}
\label{cov1}
H \circ \tau_\zeta = H , \qquad H \circ e^{\im t} = H \ , \quad \forall \varsigma \in \R , \ \ \forall t \in \R \ .  
\end{equation}
Its  Hamiltonian vector field $X_H$  thus fulfill
\begin{equation}
\label{cov2}
\tau_\zeta X_H(u) = X_H(\tau_\zeta u) , \qquad
e^{\im t } X_H(u) = e^{\im t} X_H(u) \ , \quad \forall \varsigma \in \R,  \ \forall t \in \R  \ .
\end{equation}
Consider now a quasi-periodic traveling wave $q(\f,x)$ fulfilling \eqref{cov00}, \eqref{cov01}. 
Denote by
$$
A(\f) u := d X_H\big(q(\f,\cdot) \big) u 
$$
and
$$
\cN(\f,  u) :=  X_H\big(q(\f,\cdot) + u \big)- X_H\big( q(\f,\cdot)\big)- 
d X_H\big(q(\f,\cdot) \big) u 
$$
Following \cite[Section 3.4]{BFM, BFM2}, now we prove the following lemma.
\begin{lemma}
For any $(\f, \zeta, t) \in \T^d \times \R^2 \times \R$ one has
\begin{align}
&	A(\f +  \bK \zeta) \circ \tau_\zeta = \tau_\zeta \circ A(\f) , \quad A(\f + t \vec{1}) \circ e^{\im t } = e^{\im t } \circ A(\f)
 , \label{covA.1}
\\
&	\cN(\f +  \bK \zeta, \tau_\zeta u ) = \tau_\zeta \circ \cN(\f,u) , \qquad \cN(\f + t \vec{1},  e^{\im t} u)  = e^{\im t } \circ \cN(\f,u)
  \, .  \label{covN.1}
\end{align}
\end{lemma}
\begin{proof}
	We prove just the first identities in both line, the second one being similar.
	We start with the first identity of \eqref{covA.1}.
	Differentiating  the first of \eqref{cov2} at 
	$q(\f,x)$ in direction $w$ we get
	\begin{equation}\label{cov5}
	\tau_\zeta \left[ d X_H(q(\f,\cdot) ) w \right]= d X_H(\tau_\zeta q(\f, \cdot)) \tau_\zeta w \ .
	\end{equation}
	Then  the covariance property \eqref{cov01} implies immediately \begin{equation}
	\label{cov3}
	A(\f + \bK \zeta) \circ \tau_\zeta = \tau_\zeta \circ A(\f)
	\end{equation}
	To prove \eqref{covN.1} just use again that, by \eqref{cov2},  \eqref{cov5} and the covariance properties \eqref{cov01} of the quasi-periodic traveling wave,
	\begin{align*}
	\tau_\zeta \circ \cN(\f, u) & = 
	X_H(\tau_\zeta q(\f, \cdot) + \tau_\zeta u) - X_H(\tau_\zeta q(\f, \cdot)) - d X_H(\tau_\zeta q(\f, \cdot)) \tau_\zeta w  \\
	& =
	X_H( q(\f + \bK \zeta, \cdot) + \tau_\zeta u) - X_H(q(\f+\bK \zeta, \cdot)) - d X_H(q(\f+\bK \zeta, \cdot)) \tau_\zeta w\\
	& =
	\cN(\f+ \bK \zeta,  \tau_\zeta u) \,.
	\end{align*}
\end{proof}

\newcommand{\etalchar}[1]{$^{#1}$}
\def\cprime{$'$}

	\end{document}